\documentclass[12pt,a4paper]{article}

\usepackage{amsthm,xcolor,amssymb,nicefrac,enumitem,hyperref,bm,MnSymbol}
\usepackage[margin=0.7in]{geometry}
\usepackage[sort,nocompress]{cite}

\usepackage[sort,capitalize,nameinlink,noabbrev]{cleveref}

\crefname{enumi}{item}{items}
\crefname{equation}{}{}
\crefname{figure}{Figure}{Figures}
\crefname{listing}{Source code}{Source codes}
\crefname{lstlisting}{Source code}{Source codes}
\crefname{cor}{Corollary}{Corollaries}

\hypersetup{
    colorlinks,
    linkcolor={blue!80!black},
    citecolor={green},
    urlcolor={blue!80!black}
}

\newcommand{\les}{\prec}
\newcommand{\eps}{\varepsilon}

\newcommand{\cF}{\mathcal F}

\newcommand{\cD}{\mathcal D}
\newcommand{\cE}{\mathcal E}
\newcommand{\cG}{\mathcal G}
\newcommand{\cH}{\mathcal H}

\newcommand{\cR}{\mathcal R}

\newcommand{\cU}{\mathcal U}

\newcommand{\dd}{\mathrm{d}}
\newcommand{\sfrac}[2]{\mbox{$\frac{#1}{#2}$}}

\newcommand{\IB}{\mathbb B}

\newcommand{\bX}{\mathbf X}
\newcommand{\bY}{\mathbf Y}

\newcommand{\bas}[1]{\begin{align}\begin{split} #1 \end{split}\end{align}}

\newcommand{\II}{\mathbb I}

\newcommand{\cov}{\mathrm{cov}}

\newcommand{\1}{1\hspace{-0.098cm}\mathrm{l}}

\renewcommand{\P}{{\mathbb P}}

\newcommand{\N}{{\mathbb N}}
\newcommand{\Z}{{\mathbb Z}}

\newcommand{\E}{{\mathbb E}}

\newcommand{\R}{{\mathbb R}}

\newcommand{\cK}{{\mathcal K}}
\newcommand{\cN}{{\mathcal N}}

\newcommand{\thetai}{\theta^{(i)}}
\newcommand{\sigmai}{\sigma^{(i)}}

\newcommand{\bx}{\mathbf x}
\newcommand{\by}{\mathbf y}
\newcommand{\bz}{\mathbf z}

\theoremstyle{plain}
\newtheorem{theorem}{Theorem}[section]
\newtheorem{prop}[theorem]{Proposition}
\newtheorem{lemma}[theorem]{Lemma}

\newtheorem{defi}[theorem]{Definition}

\theoremstyle{definition}

\makeatletter
\@namedef{subjclassname@2020}{%
	\textup{2020} Mathematics Subject Classification}
\makeatother

\ExplSyntaxOn

\bool_new:N \g_forexample

\NewDocumentCommand{\eg}{ o }{
\IfValueT{#1}{
\str_if_eq:noTF {fe} {#1} {
\bool_gset_true:N \g_forexample
} {\bool_gset_false:N \g_forexample}
}
\bool_if:nTF { \g_forexample } {
\bool_gset_false:N \g_forexample
for~example
}{
\bool_gset_true:N \g_forexample
for~instance
}
}

\ExplSyntaxOff

\makeatletter
\ExplSyntaxOn
\seq_new:N \g_abbrs
\prop_new:N \g_abbr_counts
\tl_new:N \l_abbr_count_tl

\NewDocumentCommand{\abbr}{m m O{#1} m m O{#4}}{
	\expandafter\newcommand\csname#3\endcsname[1][]{
		\seq_if_in:NnTF \g_abbrs {#1} {
			\prop_get:NnN \g_abbr_counts {#1} \l_abbr_count_tl
			\prop_gput:Nnx \g_abbr_counts {#1} {\int_eval:n {\l_abbr_count_tl + 1}}
			\hyperref[#1]{#1}
		} {
			\seq_gput_left:Nn \g_abbrs {#1}
			\prop_gput:Nnn \g_abbr_counts {#1} {1}
			\expandafter\gdef\csname#1@def\endcsname{#2}
			\phantomsection\label{#1}
			\str_if_eq:nnTF{##1}{}{\emph{#2}}{##1}~(\hyperref[#1]{#1})
		}
	}
	\expandafter\newcommand\csname#6\endcsname[1][]{
		\seq_if_in:NnTF \g_abbrs {#1} {
			\prop_get:NnN \g_abbr_counts {#1} \l_abbr_count_tl
			\prop_gput:Nnx \g_abbr_counts {#1} {\int_eval:n {\l_abbr_count_tl + 1}}
			\hyperref[#1]{#4}
		} {
			\expandafter\gdef\csname#1@def\endcsname{#5}
			\seq_gput_left:Nn \g_abbrs {#1}
			\prop_gput:Nnn \g_abbr_counts {#1} {1}
			\phantomsection\label{#1}
			\str_if_eq:nnTF{##1}{}{\emph{#5}}{##1}~(\hyperref[#1]{#4})
		}
	}
}

\ExplSyntaxOff
\makeatother

\abbr{ANN}{artificial neural network}{ANNs}{artificial neural networks}
\abbr{EMA}{exponential moving average}{EMAs}{exponential moving averages}
\abbr{DNN}{deep neural network}{DNNs}{deep neural networks}
\abbr{SGD}{stochastic gradient descent}{SGDs}{stochastic gradient descent}
\abbr{GD}{gradient descent}{GDs}{gradient descent}
\abbr{SOP}{stochastic optimization problem}{SOPs}{stochastic optimization problems}
\abbr{ODE}{ordinary differential equation}{ODEs}{ordinary differential equations}
\abbr{CLT}{central limit theorem}{CLTs}{central limit theorems}
\abbr{Adam}{adaptive moment estimation}{Adams}{adaptive moment estimations}
\abbr{AdamW}{\Adam\ with decoupled weight decay}{AdamWs}{\Adam\ with decoupled weight decays}
\abbr{RMSprop}{root mean square propagation}{RMSprops}{root mean square propagations}
\abbr{AI}{artificial intelligence}{AIs}{artificial intelligences}

\title{Central limit theorem for\\ the averaged Adam optimizer}

\author{Steffen Dereich$^{1}$ and Arnulf Jentzen$^{2,3}$\bigskip\\
\small{$^1$ Institute for Mathematical Stochastics, University of M\"unster,}\vspace{-0.1cm}\\
\small{Germany; e-mail: \texttt{steffen.dereich}\textcircled{\texttt{a}}\texttt{uni-muenster.de}}\smallskip\\
\small{$^2$ School of Data Science and School of Artificial Intelligence,}\vspace{-0.1cm}\\
\small{The Chinese University of Hong Kong, Shenzhen (CUHK-Shenzhen),}\vspace{-0.1cm}\\
\small{China; e-mail: \texttt{ajentzen}\textcircled{\texttt{a}}\texttt{cuhk.edu.cn}}\smallskip\\
\small{$^3$ Applied Mathematics: Institute for Analysis and Numerics,}\vspace{-0.1cm}\\
\small{University of M\"unster, Germany; e-mail: \texttt{ajentzen}\textcircled{\texttt{a}}\texttt{uni-muenster.de}}}

\date{\today}

\begin{document}

\maketitle

\begin{abstract}
In this article, we analyse convergence of the averaged Adam optimizer to an attracting zero of the Adam vector field.
We provide a central limit theorem that, in particular, quantifies exactly the speed of convergence. The  order of convergence is $n^{-1/2}$ in the number of steps of the algorithm which coincides with the order observed for classical stochastic approximation algorithms.
The covariance in the central limit theorem is given in terms of properties of the Adam algorithm in the state of the attractor.
\end{abstract}

\tableofcontents

\pagebreak

\section{Introduction}

\SGD[Stochastic gradient descent]\ optimization methods are these days the methods of choice
in the training of \AI\ systems (cf., \eg, \cite{Ruder2016arXiv,JentzenBookDeepLearning2023}).
The most popular of such \SGD\ optimization methods
is -- beyond the plain vanilla standard \SGD\ method~\cite{MR42668} -- presumably the \Adam\ optimizer proposed in 2014 by
Kingma \& Ba~\cite{KingmaBa2014_Adam}.

Despite the popularity and success of \Adam, it remains an open research problem of fundamental relevance
to establish a \emph{complete convergence analysis for \Adam}\ and it is a very active research topic
to prove rigorous error bounds for \Adam\ when applied to \SOPs.

In particular,
the recent works \cite{DereichAdamconvergence2024,DereichDoJentzen2026arXiv}
demonstrate under suitable assumptions that
the \Adam\ optimization process
$ ( \theta_n )_{ n \in \N_0 } $
satisfies that
for every number of \Adam\ steps $ n \in \N = \{ 1, 2, 3, \dots \} $ we have that the standard root mean square distance
between \Adam\ $ \theta_n $ after $ n $ steps and
the limit point $ \theta^{*} $ of \Adam\ (a zero of
the \Adam\ vector field \cite[Corollary~1.10]{DereichJentzenKassing2025arXiv}) is bounded from above by
the term $ c \sqrt{ \gamma_n } $, that is, by
a constant $ c \in (0,\infty) $ multiplied by
the square root of the learning rate (step size) $ \gamma_n $
after $ n $ steps in the sense that
\begin{equation}
\label{eq:error_Adam}
  ( \E[ \| \theta_n - \theta^{*} \|^2 ] )^{ 1 / 2 }
  \leq
  c \sqrt{ \gamma_n }
\end{equation}
(see \cite[Theorem~1.1]{DereichAdamconvergence2024}).
If the learning rates $ ( \gamma_n )_{ n \in \N } $
satisfy for all $ n \in \N $ that
$ \gamma_n = n^{ - \nu } $
for some $ \nu \in ( \nicefrac{ 1 }{ 2 }, 1 ) $,
then the root mean square error of \Adam\ on the left hand side
of \cref{eq:error_Adam} is bounded by $ c n^{ - \nicefrac{ \nu }{ 2 } } $.

For the standard \SGD\ method
the analogous statement to \cref{eq:error_Adam} (under even weaker assumptions)
is known for a long time, going back to the theory of stochastic approximations
(cf., \eg, \cite{MR42668}, \cite[Part II]{MR1082341},
\cite[Theorem~2.5]{MR3941932}, and \cite{MR4205063} and the references therein)
and in the situation where the learning rates satisfy $ \gamma_n = n^{ - \nu } $
it has also been established that the speed of convergence $ n^{ - \nicefrac{ \nu }{ 2 } } $
can essentially not be improved (cf., \eg, \cite[Theorem~1.1]{MR4055054} and the references therein).
However, for the standard \SGD\ method it is also known that
the maximal convergence speed (the highest possible asymptotic convergence speed)
$ \frac{ 1 }{ \sqrt{n} } $
can be achieved if one approximates
the limit point of the standard \SGD\ method --
a zero of the gradient (criticial point) of the objective function (the function that one intends to minimize in the considered \SOP) --
through averages of \SGD\ (usually referred to as \emph{Ruppert--Polyak averaging} \cite{Ruppert1988,Polyak1990_Automation})
and that the $ \sqrt{n} $--scaled error between
the average of \SGD\ and the critical point converges weakly
to a centered normal distribution.
We refer, \eg, to \cite{MR1167814,DerKas23_CLT,Lakshmanan2026arXiv_Survery,MoulinesBach2011_NIPS2011_40008b9a,MR4518751,MR1275360}
and the references therein for further works on this kind of error estimates
and \CLTs\ for averages of \SGD\ (cf.\ also \cref{ssec:literature} below for a short literature review).

It is the key contribution of this work to establish
this \emph{Ruppert--Polyak averaging \CLT}\ for the famous \Adam\ optimizer.
Specifically, the main result of this work -- see \cref{thm-1} below --
reveals under suitable assumptions in the notation of \cref{eq:error_Adam} that
the $ \sqrt{n} $--scaled error
\begin{equation}
\textstyle
  \sqrt{n}
  \,
  \bigl(
      \frac{
        1
      }{ n }
      [
        \sum_{ m = 1 }^n \theta_m
	  ]
    -
    \theta^{*}
  \bigr)
\end{equation}
between the average of \Adam\ $ ( \theta_m )_{ m \in \N_0 } $ and
the limit point $ \theta^{*} $ of \Adam\ converges
weakly to a centered normal distribution
as the number $ n $ of \Adam\ steps goes to infinity.
\cref{thm-1} also explicitly specifies
the covariance matrix for the limiting normal distribution
(see \cref{eq:conclusion_main_theorem} in \cref{thm-1} below).
To the best of our knowledge, this is the first work that establishes
a \CLT\ for the averaged \Adam\ optimizer.
We also refer to \cref{ssec:literature} below for a short literature review.

The remainder of this work is organized in the following way.
In \cref{sec:main_result} we recall the definition of the \Adam\ optimizer,
the \Adam\ vector field, and the \Adam\ sequence space,
we introduce a tool which we refer to as \emph{marginale correction} of \Adam,
and, thereafter, we formulate in \cref{thm-1} below the main result of this work,
which establishes under suitable assumptions the optimal convergence rate for averages of the \Adam\ optimizer.
Moreover, in \cref{sec2_1} we also provide a sketch of the proof of \cref{thm-1}.
\Cref{sec4,sec:martingale_correction,sec5,sec:proof_main_result}
are devoted to the proof of \cref{thm-1}. More specifically, a crucial ingredient
in our proof of \cref{thm-1} is the martingale correction that we introduce in \cref{def:mart_cor} below
and in \cref{sec:martingale_correction} we establish several essential properties for this martinage correction term,
particularly the local Lipschitz continuity estimate for the martingale correction in \cref{le:3451} below.
In \cref{sec4} we record several technical auxiliary results that we employ in our proof of \cref{thm-1}.
In particular,
a central idea in our proof of \cref{thm-1} is to approximate the \Adam\ vector field
(which describes the dynamics of the \Adam\ optimizer)
by a first order Taylor expansion of this \Adam\ vector field
around the limit point $ \theta^{*} $ of \Adam\ (by a linearization of this \Adam\ vector field around $ \theta^{*} $)
and in \cref{sec4} we establish convergence of the linear dynamics associated to this first order Taylor approximation
(cf. \cref{le:H-1,prop:23553})
and we also recall the convergence properties
of \Adam\ \cite[Theorem~2.5]{DereichAdamconvergence2024} that
we employ in our proof of \cref{thm-1} (cf.\ \cref{prop2352}).
In \cref{prop:98734456} in \cref{sec5} we establish a \CLT\ for a martingale
that appears due the inhomogenity in the linear dynamics associated to the first order Taylor approximation
and in \cref{sec:proof_main_result} we combine the findings from
\cref{sec:martingale_correction,sec4,sec5} to prove \cref{thm-1}.

\section{Main result: Central limit theorem (CLT) for the averaged Adam optimizer}
\label{sec:main_result}

In this section we present in \cref{thm-1} below the main convergence theorem of this work.
To do so, we first need to recall some notation. In particular,
in the following we denote by $ ( \Omega, \cF, \P ) $ a probability space,
we denote by $ d \in \N $ the dimension of the underlying \SOP,
and we recall the description of an innovation,
the \Adam\ optimizer,
the \Adam\ vector field,
the \Adam\ sequence space,
and the effective action function of \Adam\ from
\cite[Sections~2 and 3]{DereichAdamconvergence2024}.

\begin{defi}[Innovation]
\label{def:V}
A pair $(X,U)$ consisting of
\begin{enumerate}[label=(\roman*)]
\item a random variable $ U $ taking values in a measurable space $ \cU $ and
\item a product measurable mapping $ X \colon \R^d  \times \cU\to \R^d $
\end{enumerate}
is called \emph{innovation}.
\end{defi}

\begin{defi}[\Adam\ optimizer]
Consider
\begin{enumerate}[label=(\roman*)]
\item an innovation $ ( X, U ) $,
\item $ \alpha, \beta \in [0,1) $, $ \epsilon \in (0,\infty) $ with $ \alpha < \sqrt \beta $ (the \emph{damping factors}),
\item $ n_0 \in \N_0 $, $ \theta_{ n_0 }, m_{ n_0 } \in \R^d $, $ v_{ n_0 } \in [0,\infty)^d $ (the initialisation), and
\item a decreasing $(0,\infty)$-valued sequence $ ( \gamma_n )_{ n \in \N } $ (\emph{sequence of step-sizes}).
\end{enumerate}
%
%
%
%
An $\R^{d}$-valued stochastic process  $(\theta_{n})_{n\in\N_0\cap[n_{0},\infty)}$ is called \emph{\Adam\ algorithm with damping parameter $(\alpha,\beta,\epsilon)$ and step-sizes $(\gamma_{n})$ started at time $n_{0}$ in $(\theta_{0},m_{n_{0}},v_{n_{0}})$} if and only if
for every $ n \in \N \cap ( n_0, \infty) $, $ i \in \{ 1, \dots, d \} $ it holds that
\begin{equation}
  \thetai_n = \thetai_{ n - 1 } + \gamma_n \, \sigmai_n \, m^{(i)}_n,
\end{equation}
where for all $ n \in \N \cap ( n_0, \infty) $, $ i \in \{ 1, \dots, d \} $ we have
\begin{enumerate}[label=(\alph*)]
\item  $m_n=\alpha \,m_{n-1}+(1-\alpha) \,X(\theta_{n-1}, U_n)$,
\item  $v^{(i)}_n= \beta \,v^{(i)}_{n-1} +(1-\beta) \,(X^{(i)}(\theta_{n-1},U_n))^{2}$, and
\item  $\displaystyle \sigma^{(i)}_n= \bigl[ \epsilon + ( v_n^{(i)}/(1-\beta^{n}) )^{ 1 / 2 } \bigr]^{ - 1 } $
\end{enumerate}
and where $ ( U_n )_{ n \in \N \cap( n_0, \infty ) } $ is a family of independent copies of $ U $.
We call the sequence $(t_{n})_{n\in\N_{0}}$ given by
$
  t_n = \sum_{ k = 1 }^n \gamma_k
$
the \emph{training times} of the \Adam\ algorithm.
\end{defi}

The \Adam\ algorithm can be seen as a fast-slow system whose macroscopic action
is described by an \ODE\ driven by the \Adam\ vector field introduced
in \cite[Definition~2.4]{DereichAdamconvergence2024}.

\begin{defi}[\Adam\ vector field]
For a triple $(\alpha,\beta,\epsilon)$ of damping factors with $ \alpha<\sqrt\beta$ and an innovation $(X,U)$ we call the function
$
  f = ( f^{ (1) }, \dots, f^{ (d) } ) \colon \R^d \to \R^d
$
satisfying for all $ \theta \in \R^d $ that
\bas{
\textstyle
  f^{ (i) }( \theta )
  =
  ( 1 - \alpha ) \,
  \E\Bigl[
    \Bigl(
      \epsilon +
      \sqrt{
        ( 1 - \beta )
        \sum_{ k = - \infty }^0
        \beta^{ - k }
        | X^{(i)}(  \theta, U_k ) |^2
      }
    \Bigr)^{ - 1 }
    \sum_{k=-\infty}^{0} \alpha^{-k} X^{(i)}(\theta,U_k)
  \Bigr]
}
with $ ( U_k )_{ k \in - \N_0 } $
being a family of independent copies of $ U $,
the \emph{\Adam\ vector field of the innovation $ ( X, U ) $
for the damping factors $ ( \alpha, \beta, \epsilon ) $}.
\end{defi}

A useful perspective is to interpret the \Adam\ optimizer
as a recursion with delay which incorporates not only
the current innovation but the whole history of innovations \cite{DereichAdamconvergence2024}.
To employ this perspective, we recall
in the next notion
the \Adam\ sequence space and
the \Adam\ effective action function from \cite[Definition~2.3 and (19)]{DereichAdamconvergence2024}.

\begin{defi}[\Adam\ sequence space and \Adam\ effective action function]
Let $ \alpha, \beta \in [0,1) $, $ \epsilon \in (0,\infty) $
with $ \alpha < \sqrt{\beta} $ be damping factors.
\begin{enumerate}[label=(\roman*)]
\item
We denote by $ \varrho = ( \varrho_k )_{ k \in - \N_0 } \in \R^{ - \N_0 } $ the sequence
which satisfies for all $ k \in - \N_0 $ that
\begin{align}\label{def:varrho}
  \varrho_k =
  \epsilon^{ - 1 }
  ( 1 - \alpha )
  \bigl(
    \alpha^{ - k }
    +
    [ 1 - \alpha^2 / \beta ]^{ - 1 / 2 }
    \beta^{-k/2}
  \bigr)
\end{align}
and we denote by $ \ell^d_{ \varrho } $ the space of all sequences $\mathbf x=(x_{k})_{k\in-\N_{0}}\in(\R^{d})^{-\N_{0}}$ with
\bas{
\label{def:varrho2}
  \|\mathbf x\|_{\ell_{\varrho}^{d}}:=\sum_{k\in-\N_{0}} \varrho_{k}\, |x_{k}|<\infty.
}
We equip the space with the norm
$
  \left\| \cdot \right\|_{ \ell_{ \varrho }^d }
$
and call $ ( \ell^d_{ \varrho }, \left\| \cdot \right\|_{ \ell_{ \varrho }^d } ) $
the \emph{sequence space associated with the damping factors $ ( \alpha, \beta, \epsilon) $}.
\item
We denote by $ g = ( g^{ (1) }, \dots, g^{ (d) } ) \colon \ell_{ \varrho }^d \to \R^d $ function
which satisfies for all $ \mathbf x = ( x_k )_{ k \in - \N_0 } \in \ell^d_{ \varrho } $, $ i \in \{ 1, \dots, d \} $ that
\bas{
  g^{(i)}(\mathbf x)
  =
  \frac{ 1 - \alpha }{
    \epsilon
    +
    [
      (1 - \beta)
      \sum_{ k = - \infty }^0
      \beta^{ - k }
      ( x^{ (i) }_k )^2
    ]^{ 1 / 2 }
  }
  \biggl[
    \sum_{ k = - \infty }^0
    \alpha^{ - k } x^{ (i) }_k
  \biggr]
}
and we call g the \emph{effective action of the \Adam\ algorithm for the damping factors $(\alpha,\beta,\epsilon)$}.
\end{enumerate}
\end{defi}
We recall from \cite[Lemma~3.1]{DereichAdamconvergence2024} that the function $ g $
is Lipschitz continuous with Lipschitz constant $ 1 $ as a mapping on
the \Adam\ sequence space $ \ell^d_\varrho $.
We also note that the \Adam\ vector field can be expressed in terms of
the \Adam\ effective action function $ g $ by averaging
over histories with the innovation being evaluated in a frozen parameter $ \theta $.

To state our main theorem for averages of the \Adam\ optimizer,
we still need to introduce the quantities that govern the covariance
in the central limit theorem. For this we need to introduce a \emph{martingale correction term}
of the corresponding fast system. It will allow us to apply a martingale-coboundary decomposition
in the spirit of \cite{Gor69} to prove a central limit theorem.
\begin{defi}[Martingale correction]
\label{def:mart_cor}
Let $ ( \alpha, \beta, \epsilon ) $ be damping factors, let $ ( X, U ) $ be an innovation,
and let $ \varrho $ be as in \cref{def:varrho}.
For every $ \theta \in \R^d $ for which $ X(  \theta,U ) $ is integrable we define the function
\bas{
  \Psi_{\theta} \colon \ell_{ \varrho }^d \to \R^d
}
that takes $\bx=(x_k)_{k\in-\N_0}\in \ell_\varrho^d$ to
\bas{
\Psi_{\theta}(\bx):=\sum_{m=1}^{\infty} \E\bigl[g(\bX^{\bx,\theta}(m)) -f(\theta)\bigr],\label{eq:34663}}
where $f$ is the \Adam\ vector field,
where $\bX^{\bx,\theta}(m)=(\bX^{\bx,\theta}_{k}(m))_{k\in-\N_0}$ is the $\ell_{\varrho}^d$-valued random variable given by
\bas{
\bX^{\bx,\theta}_{k}(m)=\begin{cases}x_{m+k},&\text{ if }m+k\le0 , \\
X(\theta, U_{m+k}),&\text{ if } m+k>0 ,
\end{cases}
}
and where $(U_k)_{k\in\N}$ are independent copies of $U$.
$\Psi_{\theta}$ is called \emph{martingale correction in the state $\theta$}.
\end{defi}
The martingale correction is well-defined as
we will show in \cref{le:89734556} below.
Now we are in the position to state the main theorem.
\begin{theorem}[Convergence rates for the averaged \Adam\ optimizer]
\label{thm-1}
Consider
\begin{enumerate}[label=(\roman*)]
\item $ \alpha, \beta \in [0,1) $, $ \epsilon \in (0,\infty) $ with $ \alpha < \sqrt \beta $ (the \emph{damping factors}),
\item a decreasing $(0,\infty)$-valued sequence $(\gamma_{n})_{n\in\N}$ (\emph{sequence of step-sizes}), and
\item an $(X,U)$ be an innovation
\end{enumerate}
Let $ ( \theta_n )_{ n \in \N_0 } $ be an \Adam\ algorithm with innovation $ (X, U) $ and damping factors $ \alpha, \beta, \epsilon $,
let $ n_0 \in \N $ and for every $ n \in \{ n_0 + 1, \dots \} $ let
\bas{
  \bar\theta_n=\frac1{n-n_0}\sum_{k=n_0+1}^n \theta_k .
}
Moreover, denote by $f$ its \Adam\ vector field and by $\theta^{*}\in\R^{d}$ a zero of $f$. We assume the following assumptions:
\smallskip
\begin{enumerate}[label=(\alph*)]
\item \emph{Regularity of innovations.} There exist $p\in(2,\infty)$, $\cK,\delta\in(0,\infty)$,  $q\in(1-\beta,1]$
and a neighbourhood $ V $ of $ \theta^{*} $ such that
for every $ \theta, \theta' \in V $, $ i \in \{ 1, \dots, d \} $ one has
\bas{
  \E[|X(\theta,U)|^{p}]^{1/p}\le \cK,
  \qquad
  \E[|X(\theta,U)-X(\theta',U)|^{p}]^{1/p} \le \cK  \,|\theta-\theta'| ,
}
\bas{
  \text{and}
  \qquad
  \P( | X^{(i)}(\theta,U) |^2 \ge \delta )\ge q .
}
\item \emph{Stable attractor.} There exists a matrix $H\in\R^{d\times d}$  such that
\bas{
\sup\{\mathrm{Re}(\varrho):\varrho\text{ complex eigenvalue of }H\}<0
}
and there exist 
$ C \in ( 0, \infty ) $, $ \lambda \in (0,1] $ such that for all $ \theta \in V $ one has
\bas{\label{eq:237462}
|f(\theta)-H(\theta-\theta^{*})|\le C\, |\theta-\theta^*|^{1+\lambda}.
}
\item \emph{Step-sizes.} One has
\bas{\label{eq:237856-1}
\lim_{n\to\infty}\frac { \gamma_{n}- \gamma_{n+1}}{ \gamma_{n}^{2}}=0
}
and
\bas{\label{eq:237856-3}
\lim_{n\to\infty} \frac1{\sqrt n}\sum_{k=1}^n \gamma_{k}^{\frac{1+\lambda}2}=0.
}
\end{enumerate}
Let $\Gamma=(\Gamma_{i,j})_{i,j=1,\dots,d}\in\R^{d\times d}$ be the matrix with
\bas{
\Gamma_{i,j} =\E[\cD^{(i)} \cD^{(j)}],
}
where
\bas{\cD =  g(\bX^{\theta^{*}}(1))+\Psi_{\theta^{*}}(\bX^{\theta^{*}}(1))-\Psi_{\theta^{*}}(\bX^{\theta^{*}}(0)) 
}
and $\bX^{\theta^{*}}(n)=(X(\theta^{*},U_{n+k}))_{k\in-\N_{0}}$ for all $n\in\N$.
Then one has, on the event $\{\lim_{n\to\infty}\theta_{n}=\theta^{*}\}$, weak convergence
\bas{
\label{eq:conclusion_main_theorem}
\sqrt n\,(\bar \theta_{n}-\theta^{*})\Rightarrow \cN\bigl(0,H^{-1} \Gamma (H^{-1})^{\dagger}\bigr).
}
\end{theorem}

\subsection{Context discussion and literature review}
\label{ssec:literature}

In \cref{eq:conclusion_main_theorem} in \cref{thm-1} we establish a \CLT\ for the averaged \Adam\ optimizer.
In particular, we get from \cref{eq:conclusion_main_theorem} that averages of \Adam\ converge with
the optimal order of convergence $ \nicefrac{ 1 }{ 2 } $ to the limit point of \Adam\ (a zero of the \Adam\ vector field).
Our convergence analysis of the averages of \Adam\ iterates in \cref{thm-1} builds up on the error and convergence analysis
of \Adam\ iterates in our preliminary work \cite{DereichAdamconvergence2024}.
Beyond \cite{DereichAdamconvergence2024}, there are also a number of further works in the literature that
provide error estimates for \Adam\ and closely related methods;
cf., \eg, \cite[Theorem~1.2]{Dereichetal2025_Adam_symmetry_theorem_arXiv},
\cite[Theorem~4]{ReddiKale2019},
\cite[Theorem~3.1]{ZhangChen2022},
\cite[Theorem~4]{Defossez2022},
and the references therein.

A major difference of the above cited research findings to \cref{thm-1} and \cite{DereichAdamconvergence2024}, respectively, is that
the above named references provide upper bounds for certain error quantities associated to \Adam\ and closely related methods
that become smaller and converge to zero, respectively, if the number of steps of the optimizer converges to
infinity \emph{together with other optimizer parameters converging appropriately to infinity or zero},
such as the size of the mini-batch going to infinity and the distance of the second moment parameter $ \beta $
to one converging to zero, while \cref{thm-1} and \cite{DereichAdamconvergence2024} provide convergence of the error of \Adam\ to zero
if only the number of \Adam\ steps converge to zero. However,
\cref{thm-1} and \cite{DereichAdamconvergence2024} do not study convergence
to a zero of the gradient of the objective function
(a \emph{critical point} of the objective function)
but convergence to
a zero $ \theta^{*} $ of the \Adam\ vector field \cite[Corollary~1.10]{DereichJentzenKassing2025arXiv} (in this context
we also refer to the \emph{\Adam\ symmetry theorem}
in \cite[Theorem~1.1]{Dereichetal2025_Adam_symmetry_theorem_arXiv}).

Another key difference of the above named references
as well as \cite{DereichAdamconvergence2024} to \cref{thm-1} is, of course, that this work provide a convergence
analysis for the \emph{averaged} \Adam\ optimizer (instead of \Adam), which is the main purpose and difficulty of this work.
A further key difference is, of course, that \cref{eq:conclusion_main_theorem} in \cref{thm-1}
does not only offer an upper bound for the error but actually delivers an \CLT\ for the averaged \Adam\ optimizer.

Not for \Adam\ but for the standard \SGD\ optimizer and closely related methods there are also a number of works in the literature that establish
error analyses and \CLTs\ for averages of the considered optimizer; cf., \eg,
\cite{DerKas23_CLT}, \cite{MR4518751}, \cite{MR1275360}, \cite{MR1167814},
\cite[Theorem~2]{ZhuDong2020arXiv},
\cite[Sections~3.3 and 4.2]{MoulinesBach2011_NIPS2011_40008b9a},
and
\cite[Chapter~11]{MR1993642}.
We also refer to \cite[Theorem~18]{AhnCutkosky2024_arXiv} for upper bounds
for the gradient of the objective function evaluated at suitable \EMA\ averages of a clipped \Adam\ algorithm.
To the best of our knowledge, this article is the first work in the literature
that establishes a \CLT\ for the average of the \Adam\ optimizer.

We also note that the assumption in \cref{eq:237856-3}
ensures the sequence of step-sizes is a null sequence
in the sense that $ \lim_{ n \to \infty } \gamma_n = 0 $.
It is well known (see, \eg, \cite{DereichGraeberJentzen2024arXiv_non_convergence})
that \Adam\ as well as many other \SGD\ optimization methods
fail to be convergent at all if this assumption is not fulfilled
in the sense that $ \limsup_{ n \to \infty } \gamma_n \neq 0 $.

For numerical simulations for different types of averages of the \Adam\ algorithm
we refer to the reader to our
preliminary works \cite{DereichJentzenRiekert2024_arXiv_Averaged_Adam,JentzenKranzRiekert2024_arXiv_PADAM}.
Interestingly, in the performed numerical simulations it turns out that
equally weighted averages of the entire history of \Adam\ iterates
(as analyzed in \cref{eq:conclusion_main_theorem} in \cref{thm-1})
only reduce the optimization error compared to the unaveraged \Adam\ and the standard \SGD\ method
in the situation of strongly convex \SOPs\ (cf., \eg, \cite[Figures~1 and 5]{JentzenKranzRiekert2024_arXiv_PADAM})
but not in the situation of the training of \DNNs.
This empirical observation is not in contradiction to the statement of \cref{thm-1}
as \cref{thm-1} only asserts the convergence in \cref{eq:conclusion_main_theorem}
on the event where \Adam\ converges to $ \theta^{*} $ and, in particular,
where the entire \Adam\ optimization trajectory is bounded.
In the training of \DNNs\ it seems, however, that the optimization landscape
has usually the property that gradient flow, gradient desent, \SGD, and \Adam\ all
diverge to infinity (cf., \eg, \cite[Sections~1.2 and 3.3]{MR4243432}, \cite{GallonJentzenLindner2022_arXiv},
and \cite{NEURIPS2025_05fb0f4e} for details)
and in such a scenario equally weighted averages of the entire history of the iterates
of the considered optimizer seem to just slow down the divergence to infinity and, thereby, the decrease
of the objective function (cf., \eg, \cite[Figures~2--4 and 6--13]{JentzenKranzRiekert2024_arXiv_PADAM}).

However, if one changes the type of averages and replaces
equally weighted averages of the entire history of the \Adam\ iterates
by averages that give more weight on the recent \Adam\ iterates than the iterates
at the beginning of the \Adam\ process -- such as
averages of the last 1000 \Adam\ iterates or \EMAs\ of \Adam\ iterates,
then such averages often outperform the unaveraged \Adam\ and the standard \SGD\ methods
also in the training of \DNNs;
cf.\ \cite{DereichJentzenRiekert2024_arXiv_Averaged_Adam,JentzenKranzRiekert2024_arXiv_PADAM}.
In this context we also refer to \cite{Defazioetal2024_arXiv,AhnMagakyanCutkosky2024_arXiv}
for the design, numerical simulations, and error estimates for suitable schedule-free variants
of \SGD\ and \AdamW\ that
are related to \EMAs\ of \SGD\ and \AdamW\ iterates.

\subsection{Sketch of the proof of the main result}\label{sec2_1}

Let us sketch the proof of the main theorem (assuming that $\theta^*=0$). For illustration we first recall an approach that can be made rigorous when proving the classical \CLT\ for \SGD.
Let $(\theta_n^\mathrm{SGD})_{n\in\N_0}$ be a $\R^d$-valued process satisfying for each $n\in\N$
\bas{
\theta^\mathrm{SGD}_n= \theta^\mathrm{SGD}_{n-1}+\gamma_nZ_n^\mathrm{SGD},
}
where $(Z^\mathrm{SGD}_n)_{n\in\N}$ is a process  satisfying $\E[Z_n|\cF_{n-1}]=f^\mathrm{SGD}(\theta_{n-1})$ and $f^\mathrm{SGD}$ is equal to minus the gradient of the objective function. 
Suppose that in analogy to assumption \cref{eq:237462}
(with $\theta^*=0$) \bas{ f^\mathrm{SGD}(\theta)=H\theta +\mathcal O(|\theta|^{1+\lambda})}
as $\theta\to0$.
Then we can write
\bas{
\theta^\mathrm{SGD}_n=\theta^\mathrm{SGD}_{n-1}+\gamma_n \bigl(H\theta^\mathrm{SGD}_{n-1} + \underbrace{f^\mathrm{SGD}(\theta_{n-1}) -H \theta^\mathrm{SGD}_{n-1}}_{\text{error term}}+ \underbrace{Z_n^\mathrm{SGD}-f^\mathrm{SGD}(\theta_{n-1})}_{=: \Delta M^\mathrm{SGD}_n \text{ (mart. diff.)}}\bigr)
}
The error term will be negligible on $\{\theta_n\to 0\}$ and in the following we will informally write $\mathrm{err}$ for terms that depend on $n$ (and may change from line to line) that are asymptotically  negligible on $\{\theta_n\to 0\}$.
Using the variation by constant formula we use 
\bas{
\Pi[m,n]=\prod_{k=m+1}^{n}(\1+\gamma_{k}H)\text{ \ with \  $m,n\in\N_{0}$ satisfying $m\le n$}
}
to represent $\theta_n^\mathrm{SGD}$ as 
\bas{
\theta^\mathrm{SGD}_n=\Pi[0,n] \theta^\mathrm{SGD}_0 + \sum_{k=1}^n \gamma_k \Pi[k,n]\Delta M^\mathrm{SGD}_k+\mathrm{err}.
}
In terms of
\bas{\label{def:bar_Pi}
\bar\Pi[m,n]= \gamma_{m}\sum_{\ell=m}^{n}\Pi[m,\ell], \text{ \ for \  $m,n\in\N$ with $m\le n$}
}
(and $\gamma_0=1$)
we get with Fubini that
\bas{
\frac 1n \sum_{\ell=1}^n \theta_\ell^\mathrm{SGD}&=\frac 1n \sum_{\ell=1}^n\Bigl(\Pi[0,\ell] \theta^\mathrm{SGD}_0 + \sum_{m=1}^\ell \gamma_m \Pi[m,\ell]\Delta M^\mathrm{SGD}_m\Bigr)+\mathrm{err}\\
&=\underbrace{\frac 1n \bar  \Pi[0,n] \theta^\mathrm{SGD}_0}_{\text{error term}}  +\frac 1n \sum_{m=1}^n \bar \Pi[m,n]\Delta M^\mathrm{SGD}_m +\mathrm{err}
}
For most summands the matrix  $ \bar \Pi[m,n]$ is close to $-H^{-1}$ (we will use a result of \cite{DerKas23_CLT}
given in \cref{le:H-1} below) and one can represent
\bas{
\frac 1n \sum_{\ell=1}^n \theta_\ell^\mathrm{SGD}&=-\frac 1n \sum_{m=1}^nH^{-1} \Delta M^\mathrm{SGD}_m +\mathrm{err}
}
with an on $\{\theta_n^\mathrm{SGD}\to0\}$ negligible error term. Once one showed the negligibility of the error term the \CLT\ follows as consequence of a \CLT\ for martingale differences.

For the \Adam\ algorithm the situations is significantly more complex due to the \emph{memory inherent in the dynamical system}.
In that case we let $(Z_n)_{n\in\N}$ be the  $\R^d$-valued process  given by $Z_n^{(i)}=\sigma_n^{(i)} m_n^{(i)}$. Then
\bas{
\theta_n=\theta_{n-1}+\gamma_n Z_n= \theta_{n-1}+\gamma_n \bigl( H \theta_{n-1} + \underbrace{f(\theta_{n-1})-H\theta_{n-1}}_{\text{error term}} +\underbrace{Z_n-g(\bX(n))}_{\text{error term}}+g(\bX(n))\bigr).
} 
The first error term is controlled by using results of \cite{DereichAdamconvergence2024} in analogy to the analysis of the \SGD\ analysis. The second error term decays exponentially fast to zero and its contribution is easily controlled.
Using the constant of variation formula we get that
\bas{
\theta_n=\Pi[0,n] \theta_0 + \sum_{k=1}^n \gamma_k \Pi[k,n]g(\bX(k))+\mathrm{err} 
}
and
\bas{\label{eq:873456}
\frac 1n \sum_{\ell=1}^n \theta_\ell=\underbrace{\frac 1n \bar  \Pi[0,n] \theta_0}_{\text{error term}}  +\frac 1n \sum_{m=1}^n \bar \Pi[m,n] g(\bX(m)) +\mathrm{err}
}
The definition of the martingale correction is motivated by the fact that 
\bas{
\Delta M_n = g(\bX(n))+\Psi_{\theta_{n-1}} (\bX(n))-\Psi_{\theta_{n-1}}(\bX(n-1))
}
defines a sequence of martingale differences $(\Delta M_n)_{n\in\N}$. In view of~\cref{eq:873456} we get
\bas{
\frac 1n \sum_{\ell=1}^n \theta_\ell&=\frac 1n \sum_{m=1}^n \bar \Pi[m,n] g(\bX(m)) +\mathrm{err}\\
&= \frac 1n \sum_{m=1}^n \bar \Pi[m,n] \bigl(\Delta M_m -\Psi_{\theta_{m-1}} (\bX(m))+\Psi_{\theta_{m-1}}(\bX(m-1))\bigr) +\mathrm{err}.
}
Again the aim is to show that the terms following the martingale differences are asymptotically negligible. For this note that regrouping the terms yields that
\bas{
\sum_{m=1}^n& \bar \Pi[m,n] \bigl(-\Psi_{\theta_{m-1}} (\bX(m))+\Psi_{\theta_{m-1}}(\bX(m-1))\bigr)\\
&= \sum_{m=1}^{n-1}\bigl(\bar \Pi[m,n] (-\Psi_{\theta_{m-1}} (\bX(m))) +\bar \Pi[m+1,n] \Psi_{\theta_{m}}(\bX(m))\bigr)\\
&\qquad + \bar \Pi[1,n] \Psi_{\theta_0}(\bX(0)) - \gamma_n \Psi_{\theta_{n-1}}(\bX(n))
}
The latter two terms are easily seen to be negligible. However, to prove that the sum over $m$ is negligible is significantly harder. For this we carry out a perturbation analysis of the martingale correction in the state parameter (see \cref{le:3451}).
Once we have shown that indeed the contribution of these terms are asymptotically negligible we get that
\bas{
\frac 1n \sum_{\ell=1}^n \theta_\ell&= \frac 1n \sum_{m=1}^n \bar \Pi[m,n] \Delta M_m   +\mathrm{err}.
}
In the last (again technical step) we show that we can replace the process  $(\Delta M_n)$ by the process $(\Delta \tilde M_n)$ given by
\bas{
\Delta \tilde M_n=\Psi_{\theta^*}(\tilde\bX(n))- \Psi_{\theta^*}(\tilde\bX(n-1))+g(\tilde\bX(n)),
} 
where $(\tilde\bX(n))_{n\in\Z}$ is the stationary process with $\tilde\bX(n)=(X(U_{n+k},\theta^*))_{k\in-\N_0}$. Then the \CLT\ follows by applying a \CLT\ for stationary martingale differences onto the representation
\bas{
\frac 1n \sum_{\ell=1}^n \theta_\ell&= \frac 1n \sum_{m=1}^n \bar \Pi[m,n] \Delta \tilde M_m   +\mathrm{err}\\
&= -\frac 1n \sum_{m=1}^n H^{-1} \Delta \tilde M_m   +\mathrm{err}.
} 
The most parts of the article are devoted to the perturbation analysis of the martingale correction and the control of the (new) error terms.

\section{The martingale correction}
\label{sec:martingale_correction}

In this section, we analyse the martingale correction introduced in \cref{def:mart_cor}.
For this we denote by $ \alpha, \beta \in [0,1) $, $ \epsilon \in (0,\infty) $
fixed damping factors with $ \alpha < \sqrt\beta $.
Moreover, let $ (X,U) $ be a fixed innovation.

Consider for $\bx\in\ell_{\varrho}^{d}$ and $\theta\in\R^{d}$,
the $\ell_{\varrho}^{d}$-valued process $(\bX^{\mathbf x,\theta}(n))_{n\in\N_{0}}$ given by
\bas{
  \bX^{\bx,\theta}_{k}(n)
  =
  \begin{cases}
    x_{ n + k }, & \text{ if } n + k \le 0 ,
    \\
    X( \theta, U_{n+k} ), & \text{ if } n + k > 0
  \end{cases}
}
and denote for every $\theta\in\R^d$ for which $X(U,\theta)$
is integrable by $\Psi_\theta$ the martingale correction as introduced in \cref{def:mart_cor}.

\begin{lemma}
\label{le:89734556}
Let $\theta\in\R^d$ so that $X(U,\theta)$ is integrable.
The martingale correction $\Psi_\theta$ is well-defined by equation \cref{eq:34663}
and the following statements hold for all $ \bx\in\ell^d_\varrho $:
\bas{
\label{eq:8761245}|\Psi_{\theta}(\bx)|\le (1-\sqrt{\beta})^{-1}\bigl(\|\bx\|_{\ell_{\varrho}^{d}}+\|\varrho\|_{\ell_{1}} \,\E[|X(\theta,U)|]\bigr)
}
and
\bas{
\E\bigl[ g(\bX^{\bx,\theta}(1))+\Psi_{\theta}(\bX^{\bx,\theta}(1))-\Psi_{\theta}(\bx) \bigr] =f(\theta) \label{eq8923547},
}
where we use the notation introduced in \cref{def:mart_cor}.
\end{lemma}

\begin{proof} Proof of well-definedness and \cref{eq:8761245}:
Let $(U_k)_{k\in\Z}$ be a sequence of independent copies of $U$ and let for $m\in\Z$
\bas{\bX^\theta(m)=(X(\theta,U_{k+m}))_{k\in-\N_0}.
}
Note that for all $m\in\N$, $f(\theta)=\E[g(\bX^{\theta}(m))]$ and using that  $g$ is Lipschitz continuous with constant $1$ as mapping on $\ell^d_\varrho$ by \cite{DereichAdamconvergence2024} we get that
\bas{
 \E[|g(\bX^{\bx,\theta}(m))&-g(\bX^{\theta}(m))|]\le \E[\|\bX^{\bx,\theta}(m)-\bX^{\theta}(m)\|_{\ell_\varrho^d}]\\
&\le \beta^{m/2}\, \E[\|\bx- \bX^{\theta}(0)\|_{\ell_\varrho^d}] \le  \beta^{m/2} \bigl(\|\bx\| _{\ell_\varrho^d}+ \|\varrho\|_{\ell_{1}}\, \E[|X(\theta,U)|] \bigr) .
}
Consequently,
\bas{\label{eq87235456}
\sum_{m=1}^\infty |\E[g(\bX^{\bx,\theta}(m))-f(\theta)]|\le (1-\sqrt\beta)^{-1}\bigl(\|\bx\| _{\ell^d_\varrho}+ \|\varrho\|_{\ell_{1}}\, \E[|X(\theta,U)|] \bigr).
}
This proves well-definedness of the martingale correction as well as estimate~\cref{eq:8761245}. 

Proof of~\cref{eq8923547}:
One has
\bas{
\Psi_\theta(\bx)&= \sum_{m=1}^{\infty} \E\bigl[g(\bX^{\bx,\theta}(m)) -f(\theta)\bigr]\\
&= \E\bigl[g(\bX^{\bx,\theta}(1)) -f(\theta)\bigr]+\sum_{m=2}^{\infty}  \E\bigl[\E\bigl[
g(\bX^{\bx,\theta}(m)) -f(\theta)\big|U_1\bigr]\bigr]
}
Conditionally on $U_1$ the history $\bX^{\bx,\theta}(m)$  is identically distributed as $\bX^{\bx',\theta}(m-1)$ for the $U_1$-measurable choice $\bx'=\bX^{\bx,\theta}(1)$.
Consequently,
\bas{\label{eq:4764356733}
\Psi_\theta(\bx)&= \E\bigl[g(\bX^{\bx,\theta}(1)) -f(\theta)\bigr]+\sum_{m=1}^{\infty}  \E\Bigl[ \E\bigl[
g(\bX^{\bx',\theta}(m)) -f(\theta)\bigr]\big|_{\bx'=\bX^{\bx,\theta}(1)} \Bigr].
}
Note that by \cref{eq87235456}
\bas{
\sum_{m=1}^{\infty} & \E\Bigl[ \Bigl|\E\bigl[
g(\bX^{\bx',\theta}(m)) -f(\theta)\bigr]\big|_{\bx'=\bX^{\bx,\theta}(1)} \Bigr|\Bigr]\\
&=  \E\Bigl[\sum_{m=1}^{\infty} \Bigl|\E\bigl[
g(\bX^{\bx',\theta}(m)) -f(\theta)\bigr]\big|_{\bx'=\bX^{\bx,\theta}(1)} \Bigr|\Bigr]\\
&\le \E\Bigl[(1-\sqrt\beta)^{-1}\bigl(\|\bX^{\bx,\theta}(1)\| _{\ell^d_\varrho}+ \|\varrho\|_{\ell_{1}}\, \E[|X(\theta,U)|] \bigr)\Bigr]<\infty.
}
Thus we may interchange summation and expectation on the right-hand side of~\cref{eq:4764356733} and get that
\bas{
\Psi_\theta(\bx)&= \E\bigl[g(\bX^{\bx,\theta}(1)) -f(\theta)\bigr]+\E\Bigl[ \sum_{m=1}^{\infty}  \E\bigl[
g(\bX^{\bx',\theta}(m)) -f(\theta)\bigr]\big|_{\bx'=\bX^{\bx,\theta}(1)} \Bigr]\\
&= \E\bigl[g(\bX^{\bx,\theta}(1)) -f(\theta)+\Psi_\theta (\bX^{\bx,\theta}(1))\bigr].
}
This proves \cref{eq8923547}.
%
\end{proof}
%
%

\begin{prop}\label{prop:3241}Assume the setting of \cref{thm-1} and let for every $n\in\Z$,
\bas{
\bX(n)=(\1_{\{n+k>0\}} X(\theta_{n+k-1},U_{n+k}))_{k\in-\N_0}\text{ \ and \ } 
\tilde\bX(n)=(X(\theta^*,U_{n+k}))_{k\in-\N_0}.}
Moreover, let  $(A_n)_{n\in\N}$ be a predictable sequence of events such that on every $A_n$, the random variable $X(\theta,U)$ is integrable for the choice $\theta=\theta_{n-1}$, and let for all $n\in\N$
\bas{
\Delta M_n&=\1_{A_n} \bigl(g(\bX(n))-f(\theta_{n-1})+\Psi_{\theta_{n-1}}(\bX(n))-\Psi_{\theta_{n-1}}(\bX(n-1))\bigr)\text{ \ and}\\
\Delta\tilde M_n&=g(\tilde \bX(n))+\Psi_{\theta^*}(\tilde\bX(n))-\Psi_{\theta^*}(\tilde\bX(n-1))
}
Additionally, we suppose that  $\Delta M_n$ is integrable.

Then $(\Delta M_n)_{n\in\N}$ and  $(\Delta \tilde M_n)_{n\in\Z}$ are martingale differences with respect to the filtration $(\cF_n)_{n\in\Z}$ generated by the process  $(U_n)_{n\in\Z}$. Moreover,  $(\Delta \tilde M_n)_{n\in\Z}$ is a stationary $L^2$-integrable process. 
\end{prop}

\begin{proof} Let $n\in\N$. Using that $U_n$ is independent of $\cF_{n-1}$ and that $\1_{A_n}$ and $\bX(n-1)$ are $\cF_{n-1}$-measurable we conclude with \cref{le:89734556} that
\bas{
\E[\Delta M_n|\cF_{n-1}]&=\E\bigl[\1_{A_n}\bigl(g(\bX(n))-f(\theta_{n-1})+\Psi_{\theta_{n-1}}(\bX(n))-\Psi_{\theta_{n-1}}(\bX(n-1))\bigr)\big|\cF_{n-1}\bigr]\\
&=\1_{A_n}\,\bigl(\E\bigl[g(\bX^{\bx,\theta}(1))+\Psi_{\theta}(\bX^{\bx,\theta}(1))-\Psi_{\theta}(\bx)\bigr]\big|_{(\bx,\theta)=(\bX(n-1),\theta_{n-1})}\\
&\hspace{5cm} - f(\theta_{n-1})\bigr)=0.
}
In complete analogy it follows that for every $n\in\Z$, $\E[\Delta\tilde M_n|\cF_{n-1}]=0$.

For every $n\in\Z$, $\|\tilde \bX(n)\|_{\ell_\varrho^d}$ is $L^2$-integrable  and as consequence of \cref{eq:8761245} and the uniform boundedness of $g$ also $\Delta\tilde M_n$ is $L^2$-integrable for every $n\in\Z$. The stationarity of the martingale differences follows directly from the stationarity of $(U_n)_{n\in\Z}$.
\end{proof}
%
%
%

\begin{lemma}\label{le:523}
Let $\bx,\bx'\in\ell_{\varrho}^{d}$ and $\theta\in\R^{d}$ such that $X(U,\theta)$ is integrable.
 One has that
\bas{\label{eq:8761245-2}
|\Psi_{\theta}(\bx)-\Psi_{\theta}(\bx')|\le \sqrt\beta(1-\sqrt\beta)^{-1} \,\|\bx-\bx'\|_{\ell^d_\varrho}.
}
\end{lemma}

\begin{proof}
One has that
\bas{
|\Psi_{\theta}(\bx)-\Psi_{\theta}(\bx')|&=\Bigl|  \sum_{k=1}^{\infty} \E\bigl[ g(\bX^{\bx,\theta}(k))- f(\theta)\bigr]  - \sum_{m=1}^\infty  \E\bigl[g(\bX^{\bx',\theta}(k))-f(\theta)\bigr]\Bigr|\\
&=\Bigl| \sum_{k=1}^{\infty} \E\bigl[g(\bX^{\bx,\theta}(k))-g(\bX^{\bx',\theta}(k))\bigr]\Bigr|\\
&\le \sum_{k=1}^{\infty} \E\bigl[\bigl|g(\bX^{\bx,\theta}(k))-g(\bX^{\bx',\theta}(k))\bigr|\bigr]\\
&\le  \sum_{k=1}^{\infty} \beta^{k/2}\,\|\bx-\bx'\|_{\ell^d_\varrho}\le \sqrt\beta(1-\sqrt\beta)^{-1} \,\|\bx-\bx'\|_{\ell^d_\varrho}.
}
\end{proof}

In the remainder of this section, we develop an estimate for $ | \Psi_\theta(\bx) - \Psi_{\theta'}(\bx) | $ which we will need to control perturbations in the state parameter. It is the most technical part of the article. Before we provide the crucial estimate we develop an estimate that only applies to the one-dimensional ($d=1$) setting. In the multivariate setting the inequality will be applied  component-wise. For this we denote by $\ell_\varrho$ and $\|\cdot\|_{\ell_\varrho}$ the one-dimensional versions of $\ell_\varrho^d$ and $\|\cdot\|_{\ell_\varrho^d}$.

\begin{lemma}\label{le:2564}
Let, for $\bz\in\ell_{\varrho}$,
\bas{\label{def:v_star}
v^{*}(\bz)=\inf\Bigl\{(1-\beta)\sum_{k\in-\N_{0}\backslash\{r\}} \beta^{-k} z_{k}^{2}:r\in-\N_{0}\Bigr\}
}
and let $\by,\by',\bx,\bx'\in\ell_{\varrho}$ such that
\bas{ \{k\in-\N_{0}: y_{k}\not=0\text{ or } y'_{k}\not=0\}\cap \{k\in-\N_{0}: x_{k}\not=0\text{ or } x'_{k}\not=0\}=\emptyset.
}
Then one has
\bas{
|g(\mathbf y&+\bx)-g(\mathbf y+\bx')-(g(\mathbf y'+\bx)-g(\mathbf y'+\bx'))|\\
  &\le \frac{1-\alpha^2/\beta}{(1-\alpha)^2} \Bigl(2+\frac{3}{ \sqrt{\nu^{*}}}\Bigr) 
\|\by-\by'\|_{\ell_{\varrho}}
  \| \bx-\bx'\|_{\ell_{\varrho}},
}
where 
\bas{
\nu^{*}:=(v^{*}(\by)\wedge v^{*}(\by'))+ (v^{*}(\bx)\wedge v^{*}(\bx')).
}
\end{lemma}

\begin{proof}
In the proof, we let $(\varrho_{k}')_{k\in-\N_{0}}$ be given by
$$
\varrho'_{k}= \beta^{-k/2}
$$
and we consider the respective $\ell_{\varrho'}$-norm that is defined in complete analogy to the $\ell_\varrho$-norm. Since $\alpha<\sqrt\beta$ both these norms are  equivalent. 
Choose $\bar \bx$ and $\bar {\mathbf y}$ in $\ell_{\varrho}$ such  that for every $\ell\in-\N_{0}$
\bas{
\bar x_{\ell}\in\{x_{\ell},x'_{\ell}\}\text{ \ and \ }\bar y_{\ell}\in\{y_{\ell},y'_{\ell}\}
}
and $|\bar x_{\ell}|=|x_{\ell}|\vee |x'_{\ell}|$ and $|\bar y_{\ell}|=|y_{\ell}|\vee |y'_{\ell}|$.
Then
\begin{align}\begin{split}\label{eq987456}
g(\mathbf y&+\bx)-g(\mathbf y+\bx')-(g(\mathbf y'+\bx)-g(\mathbf y'+\bx'))\\
&=g(\mathbf y+\bx)-g(\mathbf y+\bar {\bx})-(g(\bar {\mathbf y}+\bx)-g(\bar{\mathbf y}+\bar\bx))\\
&\quad +g(\bar {\mathbf y}+\bx)-g(\bar {\mathbf y}+\bar \bx)-(g(\mathbf y'+\bx)-g(\mathbf y'+\bar\bx))\\
&\quad +g( {\mathbf y}+\bar\bx)-g({\mathbf y}+\bx')-(g(\bar{\mathbf y}+\bar\bx)-g(\bar{\mathbf y}+\bx'))\\
&\quad +g(\bar {\mathbf y}+\bar\bx)-g(\bar {\mathbf y}+\bx')-(g(\mathbf y'+\bar\bx)-g(\mathbf y'+\bx')).
\end{split}\end{align}
First we will derive an estimate for $|g(\mathbf y+\bx)-g(\mathbf y+\bar {\bx})-(g(\bar {\mathbf y}+\bx)-g(\bar{\mathbf y}+\bar\bx))|$.
We let for every $j\in\N_{0}$, $\bx(j)\in\ell_{\varrho}$ be given by
\bas{
x_{k}(j)=\begin{cases} x_{k}, &\text{ if } k\le -j,\\
\bar x_{k}, &\text{ else.}
\end{cases}
}
Note that by continuity of $g$ in $\ell_{\varrho}$ we have
\bas{\label{eq3567}
g(\mathbf y&+\bx)-g(\mathbf y+\bar {\bx})-(g(\bar {\mathbf y}+\bx)-g(\bar{\mathbf y}+\bar\bx))\\
&=\sum_{j=0}^{\infty}\bigl(g(\mathbf y+\bx(j))-g(\mathbf y+ {\bx(j+1)})-(g(\bar {\mathbf y}+\bx(j))-g(\bar{\mathbf y}+\bx(j+1)))\bigr).
}
We write the summands as integrals: for $j\in\N_0$, we denote by $\mathbf e_{-j}=(\1_{\{k=-j\}})_{k\in-\N_0}\in\ell_\varrho$
and write
\begin{align}\begin{split}\label{eq3567-2}
g(&\mathbf y+\bx(j))-g(\mathbf y+ {\bx(j+1)})-(g(\bar {\mathbf y}+\bx(j))-g(\bar{\mathbf y}+\bx(j+1)))\\
&=\int_{0}^{1} \bigl(\partial_{z_{-j}}g( \mathbf y+ \underbrace{{\bx(j+1)}+t (x_{-j}-\bar x_{-j})\mathbf e_{-j}}_{=:\bx(j,t)})-\partial_{z_{-j}}g( \bar{\mathbf y}+ {\bx(j+1)}+t (x_{-j}-\bar x_{-j})\mathbf e_{-j})\bigr)
  (x_{-j}-\bar x_{-j})\,\dd t.
\end{split}\end{align}
Next, we derive an estimate for $\partial_{z_{-j}}g( \mathbf y+ {\bx(j,t)})-\partial_{z_{-j}}g( \bar{\mathbf y}+ {\bx(j,t)})$ for $-j\in \II:=\{k\in-\N_{0}: x_{k}\not=0\text{ or } x'_{k}\not=0\}$. We consider in complete analogy to before $\by(\ell)\in\ell_{\varrho}$ ($\ell\in\N_{0}$) given by
\bas{
y_{k}(\ell)=\begin{cases} y_{k}, &\text{ if } k\le -\ell,\\
\bar y_{k}, &\text{ else,}
\end{cases}
}
and write
\begin{align}\begin{split}\label{eq872356}
|\partial_{z_{-j}}&g( \mathbf y+ {\bx(j,t)})-\partial_{z_{-j}}g( \bar{\mathbf y}+ {\bx(j,t)})|\\
&\le \sum_{\ell=0}^{\infty}|\partial_{z_{-j}}g( \mathbf y(\ell)+ {\bx(j,t)})-\partial_{z_{-j}}g( {\mathbf y}(\ell+1)+ {\bx(j,t)})|.
\end{split}\end{align}
Similarly as before, we have for every $\ell\in\N_{0}$
\begin{align}\begin{split}\label{eq:2462}
\partial_{z_{-j}}&g( \mathbf y(\ell)+ {\bx(j,t)})-\partial_{z_{-j}}g( {\mathbf y}(\ell+1)+ {\bx(j,t)})\\
&=\int_0^1\partial_{z_{-\ell}}\partial_{z_{-j}}g( \underbrace{\mathbf y(\ell+1)+s (y_{-\ell}-\bar y_{-\ell})\mathbf e_{-\ell} }_{=:\by(\ell,s)}+\bx(j,t))(y_{-\ell}-\bar y_{-\ell})\, \dd s.
\end{split}\end{align}
Note that by assumption and choice  of $j$, one has $y_{-\ell}=\bar y_{-\ell}$ in the case where $\ell=j$ and we restrict in the following the focus on the case where $\ell\not=j$.
In that case, we have for $\bz\in\ell_{\varrho}$
\bas{
\partial_{z_{-\ell}}\partial _{z_{-j}} g(\bz)&= -(1-\alpha)(1-\beta) \frac1{(\sqrt{v(\bz)}+\epsilon)^{2}}\frac1{\sqrt{v(\bz)}}(\alpha^{j}\beta^{\ell}z_{-\ell}+\alpha^{\ell}\beta^{j}z_{-j} )\\
&\qquad +(1-\beta)^{2} \beta^{j+\ell}\frac{3\sqrt{v(\bz)}+\epsilon}{(\sqrt{v(\bz)}+\epsilon)^{3} v(\bz)^{3/2}} z_{-j} z_{-\ell}.
}
Using that $ (1-\beta) \beta^{k}|z_{-k}|^{2}\le v(\bz)$ for $k\in\N_{0}$ we get that
\begin{align}\begin{split}\label{eq:23456}
|\partial_{z_{-\ell}}\partial _{z_{-j}} g(\bz)|&\le (1-\alpha)\sqrt{1-\beta} \epsilon^{-2}(\alpha^{j}\beta^{\ell/2}+\alpha^{\ell}\beta^{j/2} )+ (1-\beta)3\epsilon^{-2} \beta^{(j+\ell)/2}\frac{1}{ \sqrt{v(\bz)}}\\
&\le  \epsilon^{-2} \beta^{(j+\ell)/2}\Bigl(2+\frac{3}{ \sqrt{v(\bz)}}\Bigr).
\end{split}\end{align}
By definition of $\nu^{*}$, we have for every $s,t\in[0,1]$ that
\bas{
v(\by(\ell,s)+\bx(j,t))\ge \nu^{*}
}
which implies  with~\cref{eq:2462} and \cref{eq:23456} that one has
\bas{|\partial_{z_{-j}}g( \mathbf y(\ell)+ {\bx(j,t)})-\partial_{z_{-j}}g( {\mathbf y}(\ell+1)+ {\bx(j,t)})|\le \epsilon^{-2} \beta^{(j+\ell)/2}\Bigl(2+\frac{3}{ \sqrt{\nu^{*}}}\Bigr)|y_{-\ell}-\bar y_{-\ell}|
}
so that by~\cref{eq872356}
\bas{
|\partial_{z_{-j}}g( \mathbf y+ {\bx(j,t)})-\partial_{z_{-j}}g( \bar{\mathbf y}+ {\bx(j,t)})|\le \epsilon^{-2} \Bigl(2+\frac{3}{ \sqrt{\nu^{*}
}}\Bigr) \beta^{j/2}
\|\by-\bar\by\|_{\ell_{\varrho'}}.
}
Together with \cref{eq3567} and~\cref{eq3567-2} we arrive at
\bas{|g(\mathbf y+\bx)&-g(\mathbf y+\bar {\bx})-(g(\bar {\mathbf y}+\bx)-g(\bar{\mathbf y}+\bar\bx))|\\
&\le \epsilon^{-2} \Bigl(2+\frac{3}{ \sqrt{\nu^{*}
}}\Bigr) 
\|\by-\bar\by\|_{\ell_{\varrho'}}
  \|\bx-\bar \bx\|_{\ell_{\varrho'}}.
}
In complete analogy, one deduces that
\bas{|g(\bar {\mathbf y}+\bx)&-g(\bar {\mathbf y}+\bar \bx)-(g(\mathbf y'+\bx)-g(\mathbf y'+\bar\bx))|\\
&\le \epsilon^{-2} \Bigl(2+\frac{3}{ \sqrt{\nu^{*}
}}\Bigr) 
\|\bar\by-\by'\|_{\ell_{\varrho'}}
  \|\bx-\bar \bx\|_{\ell_{\varrho'}} ,
}
\bas{|g( {\mathbf y}+\bar\bx)&-g({\mathbf y}+\bx')-(g(\bar{\mathbf y}+\bar\bx)-g(\bar{\mathbf y}+\bx'))|\\
&\le \epsilon^{-2} \Bigl(2+\frac{3}{ \sqrt{\nu^{*}
}}\Bigr) 
\|\by-\bar\by\|_{\ell_{\varrho'}}
  \|\bar \bx-\bx'\|_{\ell_{\varrho'}}
}
and
\bas{
|g(\bar {\mathbf y}+\bar\bx)&-g(\bar {\mathbf y}+\bx')-(g(\mathbf y'+\bar\bx)-g(\mathbf y'+\bx'))|\\
&\le \epsilon^{-2} \Bigl(2+\frac{3}{ \sqrt{\nu^{*}
}}\Bigr) 
\|\bar\by-\by'\|_{\ell_{\varrho'}}
  \|\bar \bx-\bx'\|_{\ell_{\varrho'}}.
}
Noting that by choice of $\bar\bx$ and $\bar\by$ one has
\bas{
\|\by-\by'\|_{\ell_{\varrho'}}=\|\by-\bar \by\|_{\ell_{\varrho'}}+\|\bar \by-\by'\|_{\ell_{\varrho'}}
}
and
\bas{
\|\bx-\bx'\|_{\ell_{\varrho'}}=\|\bx-\bar \bx\|_{\ell_{\varrho'}}+\|\bar \bx-\bx'\|_{\ell_{\varrho'}}
}
we conclude with~\cref{eq987456} that
\bas{
|g(\mathbf y&+\bx)-g(\mathbf y+\bx')-(g(\mathbf y'+\bx)-g(\mathbf y'+\bx'))|\\
&\le \epsilon^{-2} \Bigl(2+\frac{3}{ \sqrt{\nu^{*}}}\Bigr) 
\|\by-\by'\|_{\ell_{\varrho'}}
  \| \bx-\bx'\|_{\ell_{\varrho'}}.
}
The statement follows by observing that generally
\bas{\epsilon^{-1} \|\cdot\|_{\ell_{\varrho'}} \le \frac{\sqrt{1-\alpha^2/\beta}}{1-\alpha}\|\cdot\|_{\ell_\varrho}.} 
\end{proof}

\begin{lemma}\label{le:345143} Let  $\beta\in(0,1)$,  $q\in(1-\beta,1]$, $\delta\in(0,\infty)$, $N\in\N\cup\{\infty\}$
and let $v^*$ be as in \cref{def:v_star}.
Moreover, let $(Z_n)_{n\in-\N_0}$ be a real-valued process that is adapted with respect to a filtration $(\cG_n)_{n\in-\N_0}$. We let
\bas{
\cR_{N}=\bigl\{\forall k\in\{-N+1,\dots,0\}: \P( Z_k^2 \ge \delta | \cG_{k-1}) \ge q \bigr\}
}
Then one has that 
\bas{
\E\Bigl[\1_{\cR_{N}} \1_{\bigl\{\sum_{k=-N+1}^0 \1_{\{Z_k^2\ge \delta\}}\ge 2\bigr\}} v^*((\1_{\{k> -N\}} Z_k)_{k\in-\N_0})^{-1}\Bigr]\le \frac\beta{(1-\beta)} \Bigl(\frac { q}{\beta-(1-q)}\Bigr)^2\, \delta^{-1}
}
and
\bas{\label{eq:783563}
\P\Bigl(\cR_{N} \cap\Bigl\{\sum_{k=-N+1}^0 \1_{\{Z_k^2\ge \delta\}}\le 1\Bigr\}\Bigr)\le  (1-q)^{N-1} (1+(N-1)q).
}
\end{lemma}

\begin{proof} We apply a coupling argument.
On a possibly enlarged filtered probability space we can define adapted Bernoulli random variables $(I_k)_{k\in-\N_0}$ such that
\bas{
\P(I_k=1|\cF_{k-1})=q \qquad \text{and} \qquad \1_{\{\P(Z_k^2\ge \delta|\cF_{k-1})\ge q\}} \delta I_k \le Z_k^2
}
for all $k\in-\N_0$. 
We let
\bas{
S=\inf\{\ell\in\N: I_{-\ell}= 1\text{ and }\exists \, r\in\{0,\dots,\ell-1\}\text{ with } I_{-r}=1\}
}
The random variable  $S-1$ is negative binomially distributed with parameters $2$ and $q$ so that (using that $\beta>1-q$) we get that
\bas{
 \E\bigl[ \beta ^{-(S-1)}\bigr]  = \Bigl(\frac {\beta q}{\beta-(1-q)}\Bigr)^2.
}
Moreover,  we have that
\bas{
\P\bigl(S-1\ge N-1\bigr) =  (1-q)^{N}+N q(1-q)^{N-1}= (1-q)^{N-1} (1+(N-1)q).
}
On the event $\cR_{N}\cap \bigl\{\sum_{k=-N+1}^0 \1_{\{Z_k^2\ge \delta\}}\ge 2\bigr\}$, we have that
$v^*((\1_{\{k> -N\}} Z_k)_{k\in-\N_0})\ge \delta (1-\beta)\beta^S$.
Consequently,
\bas{
\E\bigl[\1_{\cR_{N}} &\1_{\{\sum_{k=-N+1}^0 \1_{\{Z_k^2\ge \delta\}}\ge 2\}} v^*((\1_{\{k>-N\}} Z_k)_{k\in-\N_0})^{-1}\bigr] \le \E\bigl[ \bigl(\delta (1-\beta) \beta^{S}\bigr)^{-1} \bigr] \\
&\le \frac1{(1-\beta)\beta\delta} \,\E[\beta^{-(S-1)}]= \frac\beta{(1-\beta)} \Bigl(\frac { q}{\beta-(1-q)}\Bigr)^2\, \delta^{-1}.
}
Moreover, on the event  $\cR_{N}$ we have that 
$\sum_{k=-N+1}^0 I_k\le \sum_{k=-N+1}^0 \1_{\{Z_k^2\ge \delta\}}$ so that
\bas{
\P\Bigl(\cR_{N} \cap\Bigl\{\sum_{k=-N+1}^0 \1_{\{Z_k^2\ge \delta\}}\le 1\Bigr\}\Bigr)\le \P(S\ge N)=(1-q)^{N-1} (1+(N-1)q).
}
\end{proof}

\begin{lemma}\label{le:345143-2} Let  $\beta\in(0,1)$,  $q\in(1-\beta,1]$, $\delta\in(0,\infty)$, $N\in\{2,3,\dots\}$
and let $v^*$ be as in \cref{def:v_star}.
Moreover, let $\bY=(Y_n)_{n\in-\N_0}$ be a vector of independent real-valued $L^2$-integrable  random variables for which $Y_{-N+1},\dots,Y_0$ are identically distributed  with
\bas{\label{eq63645}
 \P(Y_0^2\ge \delta)\ge q.
}
Let $\tilde\cE'=\{\sum_{k=-N+1}^0 \1_{\{Y_k^2\ge \delta\}}\ge 2\}$. Then one has that
\bas{
\E\bigl[&\1_{\tilde \cE' }(1+ v^*(\bY)^{-1}) (\|\bY\|_{\ell_\varrho}^2+1) \bigr]\\
&\le\Bigl(1+ \frac\beta{(1-\beta)} \Bigl(\frac { q}{\beta-(1-q)}\Bigr)^2\, \delta^{-1}\Bigr) 
\,\E\bigl[\|\bY\|_{\ell_\varrho}^2+1\big|\min(Y_0^2,Y_{-1}^2)\ge \delta\bigr].
}
Moreover, for a sequence $\bY=(Y_n)_{n\in-\N_0}$ of i.i.d.\ $ L^2 $-integrable random variables with~\cref{eq63645} one has that
\bas{
\E\bigl[&(1+ v^*(\bY)^{-1}) (\|\bY\|_{\ell_\varrho}^2+1) \bigr]\le\Bigl(1+ \frac\beta{(1-\beta)} \Bigl(\frac { q}{\beta-(1-q)}\Bigr)^2\, \delta^{-1}\Bigr) 
\,\E\bigl[\|\bY\|_{\ell_\varrho}^2+1\bigr].
}
\end{lemma}

\begin{proof} 
We consider the random set $\II$ which contains the two largest indices from $\{-N+1,\dots,0\}$ for which $Z_n^2\ge \delta$ provided those indices exist (otherwise the set contains a single or none entry). Note that $\tilde \cE'=\{\#\II=2\}$ and we have that
\bas{\label{eq876238546}
\E\bigl[\1_{\tilde \cE' }(1+ v^*(\bY)^{-1}) (\|\bY\|_{\ell_\varrho}^2+1) \bigr]
=\sum_{\substack{A\subset \{-N+1,\dots,0\}:\\ \#A=2}}
\P(\II=A)  \,\E\bigl[(1+ v^*(\bY)^{-1}) (\|\bY\|_{\ell_\varrho}^2+1)\big|\II=A\bigr]
}
Now note that 
\bas{\label{eq8728345}
1+ v^*(\bY)^{-1} \text{ \ \ and \ \ } \|\bY\|_{\ell_\varrho}^2+1
}
can be viewed as componentwise decaying, resp. increasing functions on the random vector $\bY^2:=(Y_n^2)_{n\in-\N_0}$ and that the conditional distribution of $\bY^2$ given $\{\II=A\}$ is a product distribution. Hence, we can apply the FKG inequality and get that under the conditional distribution the two random variables in \cref{eq8728345} are negatively correlated meaning that
\bas{
\E\bigl[(1+ v^*(\bY)^{-1}) (\|\bY\|_{\ell_\varrho}^2+1)\big|\II=A\bigr]\le \E[1+ v^*(\bY)^{-1}|\II=A]\, \E[\|\bY\|_{\ell_\varrho}^2+1|\II=A].
}
Thus we get with \cref{eq876238546} that
\bas{
\E\bigl[\1_{\tilde \cE' }(1+ v^*(\bY)^{-1}) (\|\bY\|_{\ell_\varrho}^2+1) \bigr] \le \E\bigl[\1_{\tilde \cE' }(1+ v^*(\bY)^{-1}) \bigr] \sup_{\substack{A\subset \{-N+1,\dots,0\}:\\ \#A=2}} \E\bigl[\|\bY\|_{\ell_\varrho}^2+1\big|\II=A\bigr].
}
Since we assumed that the random variables $Y_{-N+1},\dots,Y_0$ are independent identically distributed  the latter supremum is attained for the set $A=\{0,-1\}$ so that we get that 
\bas{
\E\bigl[\1_{\tilde \cE' }(1+ v^*(\bY)^{-1}) (\|\bY\|_{\ell_\varrho}^2+1) \bigr] \le \E\bigl[\1_{\tilde \cE' }(1+ v^*(\bY)^{-1}) \bigr] \, \E\bigl[\|\bY\|_{\ell_\varrho}^2+1\big| \min(Y_0^2,Y_1^2)\ge \delta \bigr].
}
The statement follows bounding the first term on the right-hand side with \cref{le:345143}.
In the second statement one can immediately apply the FKG inequality to obtain the result.
\end{proof}

\begin{lemma}
\label{le:3451} Let    $q\in(1-\beta,1]$ and let, for $\bz\in\ell_{\varrho}$,
\bas{
v^{*}(\bz)=\inf\Bigl\{(1-\beta)\sum_{k\in-\N_{0}\backslash\{r\}} \beta^{-k} z_{k}^{2}:r\in-\N_{0}\Bigr\}.
}
There exists a constant $C\in(0,\infty)$ only depending on the damping factors and $q$  such that for every innovation $(X,U)$, $\delta\in(0,\infty)$, $\theta,\theta'\in\R^{d}$ with
\bas{
\P(X^{(i)}(\bar\theta,U)^{2}\ge \delta)\ge q\text{, \ for all \ $\bar\theta\in\{\theta,\theta'\}$ and $i\in\{1,\dots,d\}$},
}
one has for every $\bx\in \ell_{\varrho}^d$ that
\bas{
|\Psi_\theta(\bx)-\Psi_{\theta'}(\bx)|&\le C \sum_{i=1}^d  \bigl(1+(\delta^{-1/2}+ v^*(\bx^{(i)})^{-1/2}) (\|\bx^{(i)} \|_{\ell_\varrho}+ \E[X^{(i)}(\theta,U)^2]^{1/2} )\bigr)  \\
&\hspace{6cm} \E[(X^{(i)}(\theta,U)-X^{(i)}(\theta',U))^2]^{1/2} 
}
\end{lemma}

\begin{proof}
For ease of notation we first restrict attention to the case $d=1$. One has
\bas{\label{eq:32761}
&
|\Psi_\theta(\bx)-\Psi_{\theta'}(\bx)|
\\ & =\Bigl|\E\Bigl[\sum_{k=1}^{\infty} \bigl(g(X(\theta,U_{k}), \dots, X(\theta,U_{1}),\bx)-g(X(\theta,U_{k}),\dots)
\\
&\quad -\bigl(g(X(\theta',U_{k}),\dots,X(\theta',U_{1}),\bx)-g(X(\theta',U_{k}),\dots)\bigr)\bigr)\Bigr]\Bigr|\\
&\le \sum_{k=1}^{\infty} \E\bigl[\bigl| g(X(\theta,U_{k}), \dots, X(\theta,U_{1}),\bx)-g(X(\theta,U_{k}),\dots)\\
&\quad -\bigl(g(X(\theta',U_{k}),\dots,X(\theta',U_{1}),\bx)-g(X(\theta',U_{k}),\dots,X(\theta',U_{1}), X(\theta,U_{0}),\dots)\bigr)\bigr|\bigr]\\
&+\sum_{k=1}^{\infty}\E\bigl[\bigl| g(X(\theta',U_{k}),\dots,X(\theta',U_{1}), X(\theta,U_{0}),\dots) -g(X(\theta',U_{k}),\dots)\bigr|\bigr].
}
The latter sum satisfies
\bas{\label{eq:32761-3}
& \sum_{k=1}^{\infty} \E\bigl[\bigl| g(X(\theta',U_{k}),\dots,X(\theta',U_{1}), X(\theta,U_{0}),\dots\bigr) -g(X(\theta',U_{k}),\dots)\bigr|\bigr]\\
&\le \sum_{k=1}^{\infty}\E\bigl[\bigl\| (\underbrace{0,\dots,0}_{k-\text{times}}, X(\theta,U_{0})-X(\theta',U_{0}),\dots\bigr)\bigr\|_{\ell_{\varrho}^d} \bigr]\\
&\le \sum_{k=1}^{\infty} \beta^{k/2}\, \|\varrho\|_{\ell_{1}} \E[|X(\theta,U_{0})-X(\theta',U_{0})|]=\frac{\sqrt\beta}{1-\sqrt\beta} \|\varrho\|_{\ell_{1}} \E[|X(\theta,U_{0})-X(\theta',U_{0})|].
}
We consider the first sum appearing on the right hand side of~\cref{eq:32761}.
We apply \cref{le:2564}: one has for fixed $k\in\N$
\bas{\label{eq923457}
 |g(&X(\theta,U_{k}), \dots, X(\theta,U_{1}),\bx)-g(X(\theta,U_{k}),\dots)\\
&\qquad -\bigl(g(X(\theta', U_{k}),\dots,X(\theta',U_{1}),\bx)-g(X(\theta',U_{k}),\dots,X(\theta',U_{1}), X(\theta,U_{0}),\dots)\bigr)|\\
&\le \frac{1-\alpha^2/\beta}{(1-\alpha)^2}\Bigl(2+\frac3{\sqrt{V_k^{*}}}\Bigr) \beta^{k/2}\|(X(\theta, U_{k})-X(\theta',U_{k}),\dots,X(\theta,U_{1})-X(\theta',U_{1}))\|_{\ell_{\varrho}}\\
&\qquad\qquad\qquad\qquad\qquad\qquad\qquad\qquad\qquad\qquad\qquad\|\bx-(X(\theta,U_{0}),\dots)\|_{\ell_{\varrho}},
}
where 
\bas{
V_{k}^{*}&=(v^{*}(X(\theta,U_{k}), \dots, X(U_{1},\theta),0,\dots)\wedge v^{*}(X(\theta',U_{k}), \dots, X(\theta',U_{1}),0,\dots))\\
&\qquad+ \beta ^{k} (v^{*}(\bx)\wedge v^{*}(X(\theta,U_{0}),\dots)).}
Next, we will derive an estimate for the expectation of the right-hand side of~\cref{eq923457}.
Using that $(U_k)$ is an iid sequence we get that
\bas{\label{eq:24751}
\E[\|&(X(\theta,U_{k})-X(\theta',U_{k}),\dots,X(\theta,U_{1})-X(\theta',U_{1}))\|_{\ell_{\varrho}}^2\|\bx-(X(\theta,U_{0}),\dots)\|^2_{\ell_{\varrho}}]\\
&=\E\bigl[\|(X(\theta,U_{k})-X(\theta',U_{k}),\dots,X(\theta,U_{1})-X(\theta',U_{1}))\|_{\ell_{\varrho}}^2\bigr]\,\E\bigl[\|\bx-(X(\theta,U_{0}),\dots)\|^2_{\ell_{\varrho}}\bigr]\\
&\le 2\|\varrho\|_{\ell_{1}}^2\,\E\bigl[(X(\theta,U)-X(\theta',U))^2\bigr] (\|\bx\|_{\ell_{\varrho}}^2+\|\varrho\|_{\ell_{1}}^2 \E[X(\theta,U)^2])=:\Upsilon.
}
Together with Cauchy-Schwarz we thus get that  
\bas{\label{eq97356-2} \E\Bigl[
 \frac1{\sqrt{V_k^{*}}} &\|(X(\theta,U_{k})-X(\theta',U_{k}),\dots,X(\theta,U_{1})-X(\theta',U_{1}))\|_{\ell_{\varrho}}\|\bx-(X(\theta,U_{0}),\dots)\|_{\ell_{\varrho}}\Bigr]\\
 &\le \E\Bigl[\frac1{{V_k^{*}}} \Bigr]^{1/2}\E\Bigl[\|(X(\theta,U_{k})-X(\theta',U_{k}),\dots,X(\theta,U_{1})-X(\theta',U_{1}))\|_{\ell_{\varrho}}^{2}\|\bx-(X(\theta,U_{0}),\dots)\|^{2}_{\ell_{\varrho}}\Bigr]^{1/2}\\
 &\le \E\Bigl[\frac1{{V_k^{*}}} \Bigr]^{1/2} \,\sqrt{\Upsilon}.}
Next, we will apply \cref{le:345143} to provide an estimate for $ \E[\frac 1{V_{k}^{*}}] $.
For this note that
\bas{\label{eq:87e5584}
\frac 1{V_{k}^{*}}&\le \frac1{v^{*}(X(\theta,U_{k}), \dots, X(\theta,U_{1}),0,\dots)+\beta ^{k} v^{*}(\bx)}\\
&\quad +\frac 1{v^{*}(X(\theta',U_{k}), \dots, X(\theta',U_{1}),0,\dots)+\beta ^{k} v^{*}(\bx)}\\
&\quad+\frac 1{v^{*}(X(\theta,U_{k}), \dots, X(\theta,U_{1}),0,\dots)+\beta ^{k} v^{*}(X(\theta,U_{0}),\dots)}\\
&\quad+\frac 1{v^{*}(X(\theta',U_{k}), \dots, X(\theta',U_{1}),0,\dots)+\beta ^{k} v^{*}(X(\theta,U_{0}),\dots)}.
}
We apply \cref{le:345143} for $(Z_\ell)_{\ell\in-\Z}= (X(\theta,U_{\ell+k}))$ and $N=k$. Note that  for the filtration $(\cF_\ell)$ generated by $(U_\ell)$ one has $\P(X(\theta,U_{\ell})^2\ge \delta |\cF_{\ell-1}) = \P (X(\theta,U)^2\ge \delta)\ge q$ for all $\ell \in\Z$ so that the respective set $\cR_{N}$ in the lemma is indeed the full space. Consequently, one has for the event $A_k:=\bigl\{\sum_{\ell=1}^k \1_{\{X(\theta,U_\ell)^2\ge \delta\}}\ge 2\bigr\}$
\bas{
\E&\Bigl[\frac1{v^{*}(X(\theta,U_{k}), \dots, X(\theta,U_{1}),0,\dots)+\beta ^{k} (v^{*}(\bx))}\Bigr]\\
&\le \E\Bigl[\1_{A_k} \frac1{v^{*}(X(\theta,U_{k}), \dots, X(\theta,U_{1}),0,\dots)} +\1_{A_k^c} \frac 1{\beta ^{k} (v^{*}(\bx))}\Bigr]\\
&\le \frac\beta{(1-\beta)} \Bigl(\frac { q}{\beta-(1-q)}\Bigr)^2\, \delta^{-1}+ (1-q)^{k-1} (1+(k-1)q) \frac 1{\beta ^{k} (v^{*}(\bx))}\\
&\le C_1 \delta^{-1} + C_2 \frac 1{v^{*}(\bx)},
}
where $C_1:= \frac\beta{(1-\beta)} \bigl(\frac { q}{\beta-(1-q)}\bigr)^2$ and 
$C_2:=\sup_{k'\in\N} (1-q)^{k'-1} (1+(k'-1)q)\beta^{-k'}$. Note that $C_2$ is finite since 
$1-q<\beta$ by assumption. Thus we have bounded the expectation of the first expression on the right-hand side of \cref{eq:87e5584}.
In complete analogy one gets that
\bas{
\E\Bigl[&\frac1{v^{*}(X(\theta',U_k),\dots, X(\theta',U_{1}))+\beta^{k} v^{*}(\bx)}\Bigr]
\le  C_1\delta^{-1}+\frac{C_2}{ v^{*}(\bx)}.
}
Next, consider the third term. Note that $A_k^c$ and $v^{*}(X(\theta,U_{0}),\dots)$ are independent. Consequently,
\bas{
\E&\Bigl[\frac 1{v^{*}(X(\theta,U_{k}), \dots, X(\theta,U_{1}),0,\dots)+\beta ^{k} v^{*}(X(\theta,U_{0}),\dots)}\Bigr]\\
&\le \E\Bigl[\1_{A_k} \frac 1{v^{*}(X(\theta,U_{k}), \dots, X(\theta,U_{1}),0,\dots)}\Bigr]+\beta ^{k} \, \underbrace {\E\Bigl[\1_{A_k^c} \frac 1{v^{*}(X(\theta,U_{0}),\dots)}\Bigr]}_{=\P(A_k^c)\, \E\bigl[ \frac 1{v^{*}(X(\theta,U_{0}),\dots)}\bigr]}.
}
We proceed similarly as above but additionally use that \cref{le:345143} applied  with $n^*=\infty$ gives that $\E\bigl[ \frac 1{v^{*}(X(\theta,U_{0}),\dots)}\bigr]\le C_1\delta^{-1}$. This leads to
\bas{
\E\Bigl[\frac 1{v^{*}(X(\theta,U_{k}), \dots, X(\theta,U_{1}),0,\dots)+\beta ^{k} v^{*}(X(\theta,U_{0}),\dots)}\Bigr]\le C_1\delta^{-1} +C_1C_2 \delta^{-1}.
}
The fourth term on the right-hand side of \cref{eq:87e5584} satisfies the identical bound and we get that
\bas{
\E\Bigl[\frac1{V_{k}^{*}}\Bigr]\le (4C_1 +2C_2 C_1)\delta^{-1}+ 2C_2 \frac1{v^{*}(\bx)}.
}
We combine this estimate with \cref{eq923457}, \cref{eq:24751} and~\cref{eq97356-2} and  get that
\bas{
 \E\bigl[|g(&X(\theta,U_{k}), \dots, X(\theta,U_{1}),\bx)-g(X(\theta,U_{k}),\dots)\\
&\qquad -\bigl(g(X(\theta',U_{k}),\dots,X(\theta',U_{1}),\bx)-g(X(\theta',U_{k}),\dots,X(\theta',U_{1}), X(\theta,U_{0}),\dots)\bigr)|\bigr]\\
&\le C_3 (\delta^{-1/2}+ v^*(\bx)^{-1/2}) (\|\bx \|_{\ell_\varrho}+ \E[X(\theta,U)^2]^{1/2} ) \beta^{k/2} \E[(X(\theta,U)-X(\theta',U))^2]^{1/2} 
}
with the constant $C_3$ only depending on the damping factors and $q$.
Together with~\cref{eq:32761} and \cref{eq:32761-3}
we finally get that
\bas{
|\Psi_\theta(\bx)-\Psi_{\theta'}(\bx)|&\le C_4 \bigl(1+(\delta^{-1/2}+ v^*(\bx)^{-1/2}) (\|\bx \|_{\ell_\varrho}+ \E[X(\theta,U)^2]^{1/2} )\bigr)  \E[(X(\theta,U)-X(\theta',U))^2]^{1/2} 
}
with the constant $C_4$ only depending on the damping factors and $q$.
The multivariate setting follows by applying this estimate on the individual components and using the triangle inequality.
\end{proof}

\section{Technical estimates}
\label{sec4}

In this section, we prove several technical results as preparation for the proof of the main result.
We start with deducing properties that follow from assumptions~\cref{eq:237856-1} and~\cref{eq:237856-3}
on the sequence of step-sizes $ ( \gamma_n )_{ n \in \N } $.

\begin{lemma}
\label{le:28732456}
Suppose that the $(0,\infty)$-valued sequence of step-sizes $(\gamma_n)_{ n \in \N } $ is decreasing
and satisfies~\cref{eq:237856-1} and~\cref{eq:237856-3}.
\begin{enumerate}[label=(\roman*)]
\item
\label{le:28732456_item_i}
One has $\lim_{n\to\infty} \sqrt{n} \gamma_n=0$ and $\lim_{n\to\infty}n\gamma_n=\infty$.
\item
\label{le:28732456_item_ii}
For a $\N$-valued sequence $(\rho_n)$ with $\rho_n=\mathcal O(\sqrt n)$ one has that
\bas{
\lim_{n\to\infty} \frac{\gamma_{n-\rho(n)}}{\gamma_n}=1.
}
\item
\label{le:28732456_item_iii}
There exist two $ \N $-valued sequences $ ( \ell(n) )_{ n \in \N } $
and $ ( r(n) )_{ n \in \N } $ satisfying
for all $n\in\N$ that $0\le \ell(n)\le r(n)\le n$ and
\bas{
\lim_{n\to\infty} \ell(n)=\infty,
\qquad \lim_{n\to\infty} \frac{\ell(n)+n-r(n)}{n}=0
\qquad \text{and}
\qquad
\lim_{n\to\infty} t_n-t_{r(n)}=\infty.
}
\end{enumerate}
\end{lemma}

\begin{proof}
We start with the proof of \cref{le:28732456_item_i}.
As consequence of \cref{eq:237856-3} and the monotonicity of $(\gamma_n)$
we have that
\bas{
\lim_{n\to\infty} \frac 1{\sqrt n} \underbrace{\sum_{k=1}\gamma_k}_{\ge n\gamma_n}=0
}
so that $\lim_{n\to\infty}\sqrt n \gamma_n=0$. 
Next, we prove that $\lim_{n\to\infty}n\gamma_n=\infty$.
Note that \cref{eq:237856-1} implies that also
\bas{\label{eq:2357663}
\lim_{n\to\infty}\frac { \gamma_{n}- \gamma_{n+1}}{ \gamma_{n}\gamma_{n+1}}=0
}
so that for arbitrarily fixed $\eps\in(0,\infty)$ we have for sufficiently large $n$ that
\bas{\label{eq:3487658}
\frac { \gamma_{n}- \gamma_{n+1}}{ \gamma_{n}\gamma_{n+1}}\le \eps
}
which is equivalent to the inequality
\bas{\label{eq:6351}
\gamma_{n+1}\ge \frac1{1+\eps \gamma_n} \gamma_n.
}
We set $\psi_\eps(x)=\frac1{1+\eps x} x$ $(x\in(0,\infty))$ and observe that $\psi_\eps$ is monotonically increasing (at least in a neighbourhood of $0$). 
Let us consider for arbitrarily fixed $c\in(0,\infty)$ the function  $\bar \gamma_c(n)= c \frac 1n$ $(n\in\N)$. Then one has for every $n\in\N$ that 
\bas{
\bar \gamma_{n+1}=c \frac1{n+1}\ge \bar \gamma_n -c \frac 1{n^2}.
}
Conversely, we have that 
\bas{
\psi_\eps(\bar \gamma_n)=\frac1{1+\eps \bar \gamma_n} \bar \gamma_n=\bar\gamma_n- \eps c^2\frac 1{1+\eps\bar\gamma_n}\frac1{n^2}.
}
Provided that $\eps c<1$ we get existence of an $n_0(\eps,c)\in\N$ such that for every $n\ge n_0(\eps,c)$
\bas{
\bar\gamma_{n+1}\ge \psi_\eps(\bar \gamma_n).
}

Now we proceed as follows choose $c\in(0,\infty)$ arbitrarily and $\eps\in(0,c^{-1})$. Choose $n_0$ such that \cref{eq:3487658}
is satisfied for all $n\ge n_0$. Moreover, suppose that $n_0$ is so large that $\gamma_{n_0}\le \bar\gamma_{n_0(\eps,c)}$. Then we have that by induction over $k\in\N_0$
\bas{
\gamma_{n_0+k}\ge \bar\gamma_{n_0(\eps,c)+k}.
}
In particular, we get that $\liminf_{n\to\infty} n\gamma_n\ge c$. Since $c\in(0,\infty)$ can be chosen arbitrarily we showed that $\lim_{n\to\infty} n\gamma_n=\infty$.

Next, we prove \cref{le:28732456_item_ii}.
By assumption there exists $C\in(0,\infty)$ with $\rho(n)\le C\sqrt n$ for all $n\in\N$. In particular, we have that $\lim_{n\to\infty} n-\rho(n) =\infty$ and $n-\rho(n)\ge n/2$ for all but finitely many $n\in\N$. Recall that by~\cref{eq:6351}, we have that for all sufficiently large $n\in\N$, $\gamma_{n-1}\le (1+\gamma_{n-1})\gamma_n$ so that as long as $n$ is sufficiently large
\bas{
1&\le \frac{\gamma_{n-\rho(n)}}{\gamma_n}\le \prod_{k=n-\rho(n)} ^{n-1}(1+\gamma_k) \le \exp\Bigl\{ \sum _{k=n-\rho(n)} ^{n-1}\gamma_k \Bigr\} \le \exp\{ \rho(n)\,\gamma_{n-\rho(n)} \}\\
& \le \exp\bigl\{ \rho(n)\,\gamma_{\lceil n/2\rceil} \bigr\}\le \exp\bigl\{\sqrt 2 C \underbrace{\sqrt {n/2} \,\gamma_{\lceil n/2\rceil}}_{\to 0} \bigr\}\to 1\qquad (n\to\infty).
} 
This shows \cref{le:28732456_item_ii}.

It remains to prove \cref{le:28732456_item_iii}.
For every $ \varepsilon \in ( 0, \nicefrac{1}{2} ) $ we consider
\bas{
\ell_\eps(n)=\lfloor\eps n\rfloor \text{ \ \ and \ \ }r_\eps(n)=\lceil(1-\eps) n\rceil \qquad (n\in\N).
}
Then
\bas{
\lim_{n\to\infty} \ell_\eps(n)=\infty, \ \lim_{n\to\infty} \frac{\ell_\eps(n)+n-r_\eps(n)}{n}=2\eps \text{ \ and \ } \lim_{n\to\infty}\underbrace{ t_n-t_{r(n)}}_{\ge (n-r_\eps(n))\gamma_n}=\infty.
}
Applying a diagonalisation argument yields a zero-sequence $(\eps_n)$ such that the statement is true for the choice $\ell(n)=\ell_{\eps_n}(n)$ and $r(n)=r_{\eps_n}(n)$.
\end{proof}

\begin{prop}
\label{prop:7356}
We assume the setting of \cref{thm-1}.
Let $n^*\in\N$ and consider the stopping time
\bas{ T_{n^*}'=\inf\{n\ge n^*: \theta_n\not \in V \}.
}
Moreover, let for all $k\in\N$, $r(k)=\lfloor k-(\log k)^2\rfloor$, $v^*$  as in \cref{def:v_star},
\bas{
\label{eq783549061}
  \cE'_k=\{\forall i\in\{1,\dots,d\}\exists \text{distinct }&\ell_1,\ell_2\in\{n^*+1,\dots, k-1\}\text{ with }
\\&
  (X^{(i)}(\theta_{\ell_1-1},U_{ \ell_1}))^2\wedge (X^{(i)}(\theta_{\ell_2-1},U_{ \ell_2}))^2\ge \delta\} ,
}
\bas{
\label{eq783549061_B}
  \tilde \cE'_k=\{\forall i\in\{1,\dots,d\}\exists \text{distinct }&\ell_1,\ell_2\in\{r(k),\dots, k-1\}\text{ with }
\\
&
  (X^{(i)}(\theta^*,U_{ \ell_1}))^2\wedge (X^{(i)}(\theta^*,U_{\ell_2}))^2\ge \delta\}
}
and
\bas{
\label{eq783549061-3}
  \tilde \cE''_k=\bigl\{\bigl\|(\1_{\{k+\ell<r(k)\}}\,X(\theta^*,U_{k+\ell}))_{\ell\in-\N_0}\bigr\|_{\ell_\varrho^d}\le 1\bigr\} .
}
One has
\bas{
\label{eq:683461}
  \lim_{ k \to \infty }
  \gamma_{ r( k ) }
  (\log k)^2/\sqrt{\gamma_{k-1}}=0
}
and there exists $ \kappa \in (0,\infty) $ such that the following is true:
\begin{enumerate}[label=(\roman*)]
\item
\label{prop:7356:item_i}
Almost surely, on $\{T_{n^*}'=\infty\}$, all but finitely of the events $(\cE'_k)_{k\in\N}$ occur and  for all $k\in\N$ with $k>n^*$ and all  $i\in\{1,\dots,d\}$
\bas{
\E\bigl[\1_{\{T_{n^*}'=\infty\}}\1_{\cE'_k}v^*(\bX^{(i)}(k))^{-1}\bigr]\le \kappa.
\label{eq:367122346-2}}
\item
\label{prop:7356:item_ii}
Almost surely,  all but finitely many of the events $(\tilde \cE'_k)_{k\in\N}$ and $(\tilde \cE''_k)_{k\in\N}$ occur and there exists  $\kappa\in(0,\infty)$ such that for all $k\in\N$ with $r(k)\ge1$ and all  $i\in\{1,\dots,d\}$
\bas{
\E\bigl[\1_{\tilde \cE_k'}v^*(\tilde \bX^{(i)}(k-1))^{-1}\big|\cF_{r(k)-1}\bigr]\le \kappa,\label{eq:367122346}}
where $\tilde \bX(k)=(X(\theta^*,U_{k+\ell}))_{\ell\in-\N_0}$.
\end{enumerate}
\end{prop}

\begin{proof}
The statement in \cref{eq:683461} is an
immediate consequence of \cref{le:28732456_item_ii} in \cref{le:28732456}.

Next, we prove the statement in \cref{prop:7356:item_ii}. Choose  $k_0\in\N$ such that for all $k\ge k_0$ we have that $r(k)<k$. Now fix $i\in\{1,\dots,d\}$ and $k\in\N$ with $k\ge k_0$. We apply \cref{le:345143} with $N= k-r(k)$ and $(Z_\ell)=(\1_{\{\ell > -N\}}X^{(i)}(\theta^*,U_{\ell+k-1}))_{\ell\in-\N_0}$ and note that the respective event $\cR_N$ is the sure event. Consequently,
we get that
\bas{
 \E\bigl[\1_{\tilde \cE'_k} v^*((Z_\ell)_{\ell\in-\N_0})\bigr]\le \kappa:= \frac\beta{(1-\beta)} \Bigl(\frac { q}{\beta-(1-q)}\Bigr)^2\, \delta^{-1}. 
}
We note that $\cF_{r(k-1)}$ and $\1_{\tilde \cE'_k} (Z_\ell)$ are independent so that
 \cref{eq:367122346} follows by noticing that almost surely
\bas{
\E\bigl[\1_{\tilde \cE'_k} v^*(\tilde \bX^{(i)}(k-1))^{-1}\big|\cF_{r(k)-1}\bigr]\le \E\bigl[\1_{\tilde \cE'_k} v^*((Z_\ell)_{\ell\in-\N_0})\big|\cF_{r(k)-1}\bigr]= \E\bigl[\1_{\tilde \cE'_k} v^*((Z_\ell)_{\ell\in-\N_0})\bigr].
}
Moreover, we note that by \cref{le:345143}  one has
\bas{
\P\Bigl(\sum_{m=r(k)}^{k-1} \1_{\{X^{(i)}(\theta^*,U_{\ell+k-1})^2\ge \delta\}}\le 1 \Bigr)\le  (1-q)^{k-r(k)-1} (1+(k-r(k)-1)q).
}
Since the choice of $i\in\{1,\dots,d\}$ was arbitrary we get with the Lemma of Borel-Cantelli that only finitely many of the events $((\tilde \cE'_k)^c)_{k\in\N}$ occur since
\bas{
\sum_{k=k_0}^\infty \P( (\tilde \cE'_k)^c)\le d \sum_{k=k_0}^\infty  (1-q)^{k-r(k)-1} (1+(k-r(k)-1)q)<\infty,
}
where the finiteness follows since $k-r(k)$ is asymptotically equivalent to $(\log k)^2$.
Analogously, we get that almost surely all but finitely many of the events $(\tilde\cE_k'')_{k\in\N}$ occur since 
\bas{
&
\sum_{k=k_0}^\infty \P\bigl(
\bigl\|(\1_{\{k+\ell<r(k)\}}\,X(\theta^*,U_{k+\ell}))_{\ell\in-\N_0}\bigr\|_{\ell_\varrho^d} \ge 1)
\\ & \le \sum_{k=k_0}^\infty \E\bigl[
\bigl\|(\1_{\{k+\ell<r(k)\}}\,X(\theta^*,U_{k+\ell}))_{\ell\in-\N_0}\bigr\|_{\ell_\varrho^d}^2\bigr]\\
&=\sum_{k=k_0}^\infty \sum _{\substack{ \ell \in -\N_0:\\ \ell<r(k)-k}} \varrho_\ell \,\E[|X(\theta^*,U_{0})|^2]\\
&\le\rho_0\,(1-\sqrt{\beta})^{-1}\E[|X(\theta^*,U_{0})|^2] \sum_{k=k_0}^\infty \sum _{\substack{ \ell \in -\N_0:\\ \ell<r(k)-k}} \sqrt\beta^{-\ell} <\infty,
}
where the sum is finite since $k-r(k)\sim (\log k)^2$.

It remains to show \cref{prop:7356:item_i}. Fix $i\in\{1,\dots,d\}$ and $k\in\N$ with $k>n^*$ and apply \cref{le:345143} with $N=k-n^*$ and $(Z_{\ell})=(\1_{\{\ell>-N\}} X(\theta_{k+\ell-1},U_{k+\ell}))_{\ell\in-\N_0}$. Note that the respective set~$\cR_N$ contains $\{T_{n^*}'=\infty\}$ so that
\bas{
\E\bigl[\1_{\{T_{n^*}'=\infty\}}\1_{\cE'_k}v^*(\bX^{(i)}(k))^{-1}\bigr]\le \E\bigl[\1_{\cR_N}\1_{\{\sum_{m=-N+1}^{0} \1_{\{Z_m^2 \ge \delta\}}\ge 2\}}v^*((Z_\ell)_{\ell\in-\N_0})^{-1}\bigr]\le \kappa
}
which proves \cref{eq:367122346}. To see that all but finitely many $( \cE'_k)_{k_\in\N}$ occur on $\{T_{n^*}=\infty\}$ is straight-forward.
\end{proof}

The main idea of the proof is to  linearise the action of the vector field in the neighbourhood of an attractor $\theta^*$. For this we use results deduced in \cite{Der21_MR4184368} to write linearised approximations to the real algorithm. In order to show that the linearisation error is sufficiently small we need to control the distance of the algorithm to the attractor. For this we use the main result of \cite{DereichAdamconvergence2024}.
In the remaining part of this section, we prepare these results for our proofs.

First we provide a classical result of linear algebra.
\begin{lemma}\label{leno_2_exa}
Let $H\in\R^{d\times d}$, $L\in(0,\infty)$ such that
$$
\sup\{\mathrm{Re}(\lambda): \lambda\text{ e.v.\ of }H\}<-L,
$$
then there exists a norm $\|\cdot\|$ on $\R^{d\times d}$ induced by a scalar product on $\R^d$ and $\eps_0>0$ such that for all $\eps\in[0,\eps_0]$ for the induced matrix norm
\bas{\label{eq:34653}
\|\1+\eps H\|\leq 1-\eps L.
}
\end{lemma}

The proof of the main theorem is based on a linearisation of the vector field $f$ in the attractor~$\theta^*$.
For a fixed matrix  $H\in \R^{d\times d}$ and $(0,\infty)$-valued sequence $(\gamma_n)_{n\in\N}$   we consider
\bas{
\Pi[m,n]=\prod_{k=m+1}^{n}(\1+\gamma_{k}H)\text{ \ for \  $m,n\in\N_{0}$ with $m\le n$}
}
and 
\bas{
\bar\Pi[m,n]= \gamma_{m}\sum_{\ell=m}^{n}\Pi[m,\ell], \text{ \ for \  $m,n\in\N$ with $m\le n$}
}
in complete analogy to \cref{sec2_1}.

We cite an implication  of \cite[Lemma 3.6]{Der21_MR4184368}.
\begin{lemma}
\label{le:H-1}
Let  $\|\cdot\|$ be a norm on $\R^{d\times d}$ and suppose that the following is true:
\begin{enumerate}
\item[(L0)] $ \lim_{n\to\infty} \gamma_n=0$,
\item[(L1)] $ \exists \, L,\eps_{0}\in(0,\infty)$  $\forall \eps\in[0,\eps_0]:$ $\|\1+\eps H\|\leq 1-\eps L$ and 
\item[(L3)]  $\displaystyle{\lim_{n\to\infty} \frac{\gamma_n-\gamma_{n+1}}{\gamma_n^2}=0.}$
\end{enumerate}
Then the matrices $(\bar\Pi[m,n])_{1\le m\le n}$ converge to $-H^{-1}$ when letting
 $m,n\to\infty$ with $t_n-t_m\to\infty$. Furthermore,  the matrices $(\bar\Pi[m,n])_{1\le m\le n}$  are uniformly bounded.
 \end{lemma}

\begin{proof}
First note that (L3) implies that $\lim_{n\to\infty}\frac {\gamma_{n+1}}{\gamma_n}=1$ so that (L3) also entails that \bas{
\lim_{n\to\infty} \frac{\gamma_n-\gamma_{n+1}}{\gamma_n\gamma_{n+1}}=0.
}
We now apply \cite[Lemma 3.6]{Der21_MR4184368} with the sequence $(b_{n})_{n\in\N}\equiv1$  so that $\cH[m,n]=\Pi[m,n]$ and $\bar \cH[m,n]=\bar \Pi[m,n]$. The lemma implies the first statement and the second statement for all matrices $(\bar\Pi[m,n])_{n_{0}\le m\le n}$ with $n_{0}\in\N$ being sufficiently large. To get the statement as in the lemma is straight-forward.
\end{proof}

\begin{prop}
\label{prop:23553}
Under the assumptions of \cref{thm-1} there exists $R,c_1\in(0,\infty)$ and a norm $\|\cdot\|$ on $\R^d$ induced by a scalar product $\llangle\cdot,\cdot\rrangle$ such that for all $x\in \overline {\IB(\theta^*,R)}$ one has
\bas{\label{eq:352345}
\llangle f(x)-f(\theta^*),x-\theta^*\rrangle \le -c_1\|x-\theta^*\|^2.
}
The matrices $(\bar\Pi[m,n])_{1\le m\le n}$ (as defined in~\cref{def:bar_Pi}) converge to $-H^{-1}$ when letting
 $m,n\to\infty$ with $t_n-t_m\to\infty$. Furthermore,  the matrices $(\bar\Pi[m,n])_{1\le m\le n}$  are uniformly bounded.
\end{prop}

\begin{proof}
Under the assumption of \cref{thm-1}, \cref{leno_2_exa} is applicable and we can pick $L,\eps_0\in(0,\infty)$ and a norm $\|\cdot\|$ on $\R^d$ being induced by a scalar product $\llangle\cdot,\cdot\rrangle$ such that~\cref{eq:34653} is satisfied for all $\eps\in[0,\eps_0]$ and the second part of the statement is an immediate consequence of \cref{le:H-1}.

To prove the inequality we first observe that for  all $y\in\R^d\backslash\{0\}$ and $\eps\in[0,\eps_0]$
\bas{
\|y\|^2 -2\eps L \|y\|^2+\eps^2 \|y\|^2 =(1-\eps L)^2\|y\|^2\ge  \|(\1+\eps H)y\|^2=\|y\|^2+2\eps \llangle y,  H y \rrangle+\eps^2 \|Hy\|^2.
}
Letting $\eps$ go to zero it follows that $-L\|y\|^2\ge \llangle y,  H y \rrangle$. We choose $R\in(0,\infty)$ such that for all  $x\in \overline{\IB(\theta^*,R)}$ 
\bas{
\|f(x)-f(\theta^*)-H (x-\theta^*)\|\le \sfrac L2\|x-\theta^*\|.
}
Then it follows that
\bas{
\llangle f(x)-f(\theta^*),x-\theta^*\rrangle\le \llangle H(x-\theta^*),x-\theta^*\rrangle+\sfrac L2\|x-\theta^*\|^2\le -\sfrac L2 \|x-\theta^*\|^2
}
and the inequality holds for $c_1:=\frac L2>0$.
\end{proof}

Recall that in \cref{thm-1} we consider an \Adam\ algorithm $(\theta_n)_{n\in\N_0}$ with innovation $(X,U)$ and we use for the generation of $(\theta_n)$ independent copies $(U_n)_{n\in\N}$ of $U$. We interpret the \Adam\ algorithm as a delay equation and record the full history of $X$-values: we let for all $n\in\Z$,
\bas{
X_n=\begin{cases} X(\theta_{n-1},U_n), &\text{ if } n\in\N,\\
0, &\text{ else,}\end{cases} \quad \text{ \ \ and \ \  } \quad  \bX(n)=(X_{n+k})_{k\in-\N_0}.
}
We will use an error bound provided in \cite{DereichAdamconvergence2024} to control the error of the \Adam\ algorithm $(\theta_n)$ in the case where the scheme tends to an attractor $\theta^*$ of the Adam field. More explicitly, we will use the following consequence of  \cite[Theorem 2.5]{DereichAdamconvergence2024}.

\begin{prop}
\label{prop2352}
Assume the setting of \cref{thm-1}. Choose  $R,c_1\in(0,\infty)$ and a norm $\|\cdot\|$ on $\R^d$ induced by a scalar product $\llangle\cdot,\cdot\rrangle$ as in \cref{prop:23553} and assume that  $\overline{\IB(\theta^*,R)}\subset V$. For every $n^*\in\N_0$ let
\bas{\cE_{n^*}:=\{\theta_{n^*}\in \overline {\IB(\theta^*,R)} \}\cap \{\|\bX(n^{*})\|_{\ell_{\varrho}^{d}}\le (\gamma_{n^*})^{-1/2}\}}
and
\bas{T_{n^*}':=\inf\bigl\{n\in\N\cap [n^*,\infty): \|\theta_n-\theta^*\|>R\bigr\}.}
Then there exists $n_0^*\in\N$ such that for every $n^*\ge n_0^*$ there exist $\eta\in(0,\infty)$ so that for all $n\in[n^*,\infty)\cap\N_0$
\bas{
\E[\1_{\cE_{n^*}}\1_{\{T'_{n^*}\ge n\}} \|\theta_n-\theta^*\|^p]^{1/p} \le \eta \sqrt{\gamma_{n+1}}.
}
\end{prop}

\begin{proof}
We apply \cite[Theorem 2.5]{DereichAdamconvergence2024}. For this we choose $\alpha,\beta,\epsilon, (\gamma_n),p$ and $V$ as in our main theorem. The norm $\|\cdot\|$, the scalar product $\llangle\cdot,\cdot\rrangle$, $R$ and $c_1$ are chosen as in \cref{prop:23553} and $c_2\in(0,c_1)$ is fixed arbitrarily. We apply Theorem 2.5 of \cite{DereichAdamconvergence2024} with $\cK\vee R\vee 1$ in place of $\cK$ onto the randomly initialised system at time $n^*\ge n_0^*:=\mathfrak n$ (where $\mathfrak n\in\N_0$ is as in Theorem~2.5) on the event $\cE_{n*}$.

We go through the requirements (i) to (iv) of \cite[Theorem~2.5]{DereichAdamconvergence2024} step by step. (i) is satisfied by assumption and (ii) holds on $\cE_{n^*}$ by definition of the latter event. In (iii) we choose $\Psi\equiv \theta^*$ and note that the assumptions are satisfied by choice of the respective $\cK$. (iv) is satisfied as consequence of \cref{prop:23553} for $\mathfrak R\equiv R$.

Consequently, \cite[Theorem~2.5]{DereichAdamconvergence2024} implies that for every $n^*\ge n^*_0$ there exists $\eta\in(0,\infty)$ such that
\bas{
\E[\1_{\cE_{n^*}}\1_{\{T'_{n^*}\ge n\}} \|\theta_n-\theta^*\|^p]^{1/p} \le \eta \sqrt{\gamma_{n+1}}.
}
Indeed, one even has the inequality conditionally on $\cF_{n^*}$. 
\end{proof}

\begin{prop}\label{prop:9873456} Let  $R,c_1\in(0,\infty)$ and  $\|\cdot\|$ a norm on $\R^d$ induced by a scalar product $\llangle\cdot,\cdot\rrangle$ as in \cref{prop:23553}. Additionally suppose that   $\overline{\IB(\theta^*,R)}\subset V$.
Let $n^*\in\N$
and consider the stopping time
\bas{
T_{n^*}:=\inf\bigl\{n\in\N\cap [n^*,\infty): \|\theta_n-\theta^*\|>R\bigr\}\wedge \inf \bigl\{n\in\N : | X(\theta^*,U_n)|\ge n^*(n^*+n)\bigr\}.
} 
and the events
\bas{\cE_{n^*}:=\{\theta_{n^*}\in \overline {\IB(\theta^*,R)} \}\cap \{\|\bX(n^{*})\|_{\ell_{\varrho}^{d}}\le (\gamma_{n^*})^{-1/2}\}\text{ \ and \ }
\Omega_{n^*}=\cE_{n^*}\cap \{ T_{n^*}=\infty\}.} One has
\bas{\lim_{n^*\to\infty}\P\bigl(\bigl\{\lim_{n\to\infty}\theta_n=\theta^*\bigr\}\backslash \Omega_{n^*}\bigr)=0.}
\end{prop}
\begin{proof} 
Consider
\bas{T_{n^*}':=\inf\bigl\{n\in\N\cap [n^*,\infty): \|\theta_n-\theta^*\|>R\bigr\}.}
One has
\bas{
\Omega_0:=\bigl\{\lim_{n\to\infty}\theta_n=\theta^*\bigr\}\subset \bigcup _{n^*\in\N} \{T'_{n^*}=\infty\}
}
so that by monotone convergence
\bas{\label{eq823456}
\lim_{n^*\to\infty} \P\bigl(\Omega_0\backslash \{ T_{n^*}'=\infty\}\bigr)=0.
}
Note that by Jensen's inequality one has
\bas{
\P( \exists \, n \in \N \colon |X(\theta^*,U_n)|\ge n^*(n^*+n))\le \sum_{n=1}^\infty \frac 1{(n^*(n^*+n))^2} \E[|X(\theta^*,U)|^2]
}
and the latter sum has a finite limit and each summand is for each fixed $n$ monotonically decreasing in $n^*$ with limit zero. Consequently, dominated convergence ensures that
\bas{
\lim_{n^*\to\infty} \P(\exists \, n\in\N \colon |X(\theta^*,U_n)|\ge n^*(n^*+n))=0
}
and together with~\cref{eq823456} we conclude that
\bas{\label{eq823456-2}\lim_{n^*\to\infty }\P(\Omega_0\backslash \{T_{n^*}=\infty\})= 0.
}
Now let $n^*,n^{**}\in\N$ with $n^{**}\le n^*$. Note that
\bas{
\|\bX(n^*)\|_{\ell_\varrho^d}\le \sqrt{\beta}^{n^*-n^{**}} \|\bX(n^{**})\|_{\ell^d_\varrho}+ \sum_{k=n^{**}+1}^{n^*} \varrho_{k-n^*} |X(k)|
}
and
\bas{
\E\Bigl[ \sum_{k=n^{**}+1}^{n^*}\1_{\{T'_{n^{**}}\ge k\}} \varrho_{k-n^*} |X(k)|\Bigr]\le \cK\, \|\varrho\|_{\ell_1},
}
with the constant $\cK$ being as in \cref{thm-1}.
Together with the Markov inequality we get that
\bas{
\lim_{n^*\to\infty} \gamma_{n^*}^{1/2}\, \1_{\{T_{n^{**}}'=\infty\}}\|\bX(n^*)\|_{\ell_\varrho^d}=0, \text{ in probability},
}
which entails that
\bas{
\limsup_{n^*\to\infty} \P(\cE_{n^*}^c\cap \Omega_0) \le \P( \Omega_0\cap \{T'_{n^{**}}<\infty\}).
}
Note that the left-hand side does not depend on the choice of $n^{**}$ and by sending $n^{**}$ to infinity we get with~\cref{eq823456} that $\limsup_{n^*\to\infty} \P(\cE_{n^*}^c\cap \Omega_0) =0$. The statement follows when combining this with~\cref{eq823456-2}.
\end{proof}

\section{A martingale CLT}
\label{sec5}

To prove the main theorem we will relate the averaged system $\bar \theta_n-\theta^*$
to a significantly simpler martingale for which classical martingale \CLTs\ can be applied.
We use the concept of stable convergence as can be found in~\cite{HallHeyde80}.
For a short introduction concerning stable convergence on events one may consult \cite[Appendix~A]{DerKas23_CLT}.

We will prove the following.

\begin{prop}
\label{prop:98734456}
Assume the setting of \cref{thm-1} and let $(\bar \Pi[m,n])_{1\le m\le n}$ be as in~\cref{def:bar_Pi}.
We denote for every $n\in\Z$
\bas{
\tilde\bX(n)=(X(\theta^*,U_{n+k}))_{k\in-\N_0} \qquad \text{and} \qquad
\Delta \tilde M_n=g(\tilde \bX(n))+\Psi_{\theta^*}(\tilde \bX(n))-\Psi_{\theta^*}(\tilde \bX(n-1)),
}
where $\Psi$ is as in~\cref{eq:34663}. The matrix $\Gamma=(\Gamma_{i,j})_{i,j=1,\dots,d}\in\R^{d\times d}$ given by $\Gamma_{i,j}=\E[\Delta\tilde M_1^{(i)}\Delta\tilde M_1^{(j)}]$ is well-defined and one has for every $n_0\in\N$ that
\bas{
\sqrt n\,\frac1{n-n_0}\sum_{k=n_0+1}^n \bar\Pi[k,n]\,\Delta \tilde M_k\stackrel{\mathrm{stably}}{\Longrightarrow} \cN(0,H^{-1} \Gamma (H^{-1})^\dagger), \text{ \ as }n\to\infty.
}
\end{prop}

\begin{proof}[Proof of \cref{prop:98734456}]
Let $(\cF_k)_{k\in\Z}$ be the filtration generated by the process $(U_k)_{k\in\Z}$.
By \cref{prop:3241}, $(\Delta \tilde M_k)_{k\in\Z}$ is a $L^2$-integrable, stationary process of martingale differences. In particular, $\Gamma$ as in the proposition is well-defined.
We will apply a classical \CLT\ to prove the proposition.

By \cite[Theorem~A.5]{DerKas23_CLT} (see also \cite[Corollary~3.1]{HallHeyde80} for a one-dimensional version of the result),
the statement of the proposition follows once we verified the following:
\begin{enumerate}[label=(\roman*)]
\item
\label{prop:98734456_item_i}
for all $\varepsilon\in(0,\infty)$: $\lim_{n\to\infty}\sum_{k=n_0+1}^n \E\bigl[ \1\{|Z_k^{(n)}|>\varepsilon\} |Z_k^{(n)}|^2\big|\cF_{k-1}\bigr]=0$, in probability, and
\item
\label{prop:98734456_item_ii}
$\lim_{n\to\infty} \sum_{k=n_0+1}^n\cov(Z_k^{(n)}|\cF_{k-1})=\tilde \Gamma$, in probability,
\end{enumerate}
where
\bas{
Z_k^{(n)}= \frac{\sqrt n}{n-n_0} \bar \Pi[k,n] \Delta \tilde M_k\text{ \ and \ } \tilde \Gamma=H^{-1} \Gamma (H^{-1})^\dagger.
}
By \cref{prop:23553},  $C=\sup_{1\le k\le n}|\bar \Pi[k,n]|$ is finite and
for $n\ge 2n_0$, we have $\sqrt n/(n-n_0)\le 2n^{-1/2}$. Consequently,
\bas{
\1\{|Z_k^{(n)}|>\varepsilon\} |Z_k^{(n)}|^2\le
\1\{2C n^{-1/2} |\Delta \tilde M_k|>\eps\} C^2 \frac{2}{n-n_0} |\Delta \tilde M_k|^2
}
so that
\bas{
\E\Bigl[\sum_{k=n_0+1}^n &\E\bigl[ \1\{|Z_k^{(n)}|>\varepsilon\} |Z_k^{(n)}|^2\big|\cF_{k-1}\bigr]\Bigr]=\sum_{k=n_0+1}^n \E\bigl[ \1\{|Z_k^{(n)}|>\varepsilon\} |Z_k^{(n)}|^2\bigr]\\
&\le  \frac{2C^2}{n-n_0}\sum_{k=n_0+1}^n \E\bigl[ \1\{|\Delta \tilde M_k|>\sfrac1{2C}\eps n^{1/2}\} |\Delta \tilde M_k|^2\bigr]\\
&= 2C^2 \,\E\bigl[ \1\{|\Delta \tilde M_0|>\sfrac1{2C}\eps n^{1/2}\} |\Delta \tilde M_0|^2\bigr],
}
where we used stationarity of $(\Delta\tilde M_k)$ in the latter step. By $L^2$-integrability of $\Delta \tilde M_0$, we get that 
\bas{
\lim_{n\to\infty}\E\Bigl[\sum_{k=n_0+1}^n &\E\bigl[ \1\{|Z_k^{(n)}|>\varepsilon\} |Z_k^{(n)}|^2\big|\cF_{k-1}\bigr]\Bigr]
=0
}
so that property \cref{prop:98734456_item_i}
follows as consequence of the Markov inequality.

Next, we prove \cref{prop:98734456_item_ii}. First note that
\bas{
\sum_{k=n_0+1}^n\cov(Z_k^{(n)}|\cF_{k-1})&=\sum_{k=n_0+1}^n \frac{n}{(n-n_0)^2}  \,\cov(\bar\Pi[k,n] \,\Delta\tilde M_k|\cF_{k-1}) \\
&=\frac{n}{(n-n_0)^2} \sum_{k=n_0+1}^n \bar\Pi[k,n] \,\cov(\Delta \tilde M_k|\cF_{k-1}) \,\bar\Pi[k,n]^\dagger.
}
By \cref{le:28732456}, we can choose $\N$-valued sequences $(\ell(n))_{n\ge n_0+1}$ and $(r(n))_{n\ge n_0+1}$ satisfying for all $n\ge n_0+1$ that $n_0<\ell(n)\le r(n)\le n$ and
\bas{
\lim_{n\to\infty} \ell(n)=\infty, \ \lim_{n\to\infty} \frac{\ell(n)+n-r(n)}{n}=0 \text{ \ and \ } \lim_{n\to\infty} t_n-t_{r(n)}=\infty.
}
Note that
\bas{\label{eq982346}\frac 1n\E\Bigl[\Bigl| \sum_{k=n_0+1}^{\ell(n)} \bar\Pi[k,n] \,\cov(\Delta \tilde M_k|\cF_{k-1}) \,\bar\Pi[k,n]^\dagger\Bigr|\Bigr]&\le \frac{C^2}n \sum_{k=n_0+1}^{\ell(n)} \E \bigl[|\cov(\Delta \tilde M_k|\cF_{k-1})| \bigr]\\
&=\frac{C^2(\ell(n)-n_0)}n\E \bigl[|\cov(\Delta\tilde M_1|\cF_{0})| \bigr],
}
where we used stationarity of $(\Delta \tilde M_k)$ in the latter step. Since $\Delta \tilde M_1$ has finite second moments and $\ell(n)/n\to0$ as $n\to\infty$ we have that the left-hand side of~\cref{eq982346} tends to zero.
 In complete analogy one gets that 
\bas{\label{eq93562}\lim_{n\to\infty} \frac 1n\E\Bigl[\Bigl| \sum_{k=r(n)+1}^{n} \bar\Pi[k,n] \,\cov(\Delta \tilde M_k|\cF_{k-1}) \,\bar\Pi[k,n]^\dagger\Big|\Bigr]=0.
}
By \cref{prop:23553}, $\varepsilon(n):=\sup _{k\in\{\ell(n)+1,\dots ,r(n)\}} |\bar \Pi[k,n]+H^{-1}|$ tends to zero when letting $n\to\infty$. Thus
\bas{\label{eq87352}
\E\Bigl[&\Bigl|\frac1n \sum_{k=\ell(n)+1}^{r(n)} \bar\Pi[k,n] \,\cov(\Delta \tilde M_k|\cF_{k-1}) \,\bar\Pi[k,n]^\dagger\\
&\qquad -\frac1n \sum_{k=\ell(n)+1}^{r(n)} (-H)^{-1} \,\cov(\Delta \tilde M_k|\cF_{k-1}) \,(-H^{-1})^\dagger\Bigr|\\
&\le \frac1n \sum_{k=\ell(n)+1}^{r(n)} \bigl| \bar\Pi[k,n]+H^{-1}\bigr| \,\E\bigl[\bigl|\cov(\Delta \tilde M_k|\cF_{k-1})\bigr|\bigr] \,\bigl|(H^{-1})^\dagger\bigr|\\
&\qquad + \frac1n \sum_{k=\ell(n)+1}^{r(n)} \bigl| \bar\Pi[k,n]\bigr| \,\E\bigl[\bigl|\cov(\Delta \tilde M_k|\cF_{k-1})\bigr|\bigr] \,\bigl|(\Pi[k,n]+H^{-1})^\dagger\bigr|\\
&\le 2 C \eps(n) \E\bigl[\bigl|\cov(\Delta \tilde M_1|\cF_{0})\bigr|\bigr] 
}
tends to zero as $n\to\infty$. Combining estimates~\cref{eq982346}, \cref{eq93562} and \cref{eq87352} with the Markov inequality we conclude that
\bas{\label{eq823453} \lim_{n\to\infty} \Bigl| \sum_{k=n_0+1}^n\cov(Z_k^{(n)}|\cF_{k-1})
- \frac n{(n-n_0)^2}  \sum_{k=\ell(n)+1}^{r(n)} H^{-1} \cov(\Delta \tilde M_k|\cF_{k-1}) (H^{-1})^\dagger \Bigr| =0, \text{ in probability.}
}
Now using the martingale property of $(\Delta \tilde M_k)$ we conclude that for every $k\in\N$
\bas{
\E[\cov(\Delta \tilde M_k|\cF_{k-1})]=\bigl(\E\bigl[\E[\Delta \tilde M_k^{(i)}\Delta \tilde M^{(j)}_k|\cF_{k-1}]\bigr]\bigr)_{i,j=1,\dots,d}= \bigl(\E[\Delta \tilde M_k^{(i)}\Delta \tilde M^{(j)}_k]\bigr)_{i,j=1,\dots,d}=\Gamma
}
so that as consequence of the ergodic theorem
\bas{\lim_{n\to\infty} \frac 1{r(n)-\ell(n)}\sum_{k=\ell(n)+1}^{r(n)} \cov(\Delta \tilde M_k|\cF_{k-1})=\Gamma, \text{ in probability.}
} 
(Actually the convergence also occurs almost surely.) Together with~\cref{eq823453}
and the fact that
\begin{equation}
  \lim_{ n \to \infty } \frac{ n }{ (n - n_0)^2 } ( r(n) - \ell(n) ) = 1
\end{equation}
we verified property \cref{prop:98734456_item_ii}. This finishes the proof.
\end{proof}

\section{Proof of the main theorem}
\label{sec:proof_main_result}

In this section, we prove \cref{thm-1}.
The strategy of the proof is as follows. We view the averaged \Adam\ iterates as perturbed linear system.
On the event $\Omega_0:=\{\lim_{n\to\infty}\theta_n=\theta^*\}$, the martingale analysed in \cref{sec5}   provides the leading contribution, while all remaining terms are shown to be of smaller order. For the understanding of the proof, it is crucial to under stand the
 first step,  where  we write the averaged system in a linearised form. 
Without loss of generality, we assume that $\theta^*=0$. For a sketch of the proof please consult \cref{sec2_1}.

\emph{1. Viewing the Adam iterates as perturbed linear system.}
For every $n\in\N$ we can write 
\bas{\label{eq:intro}
\theta_{n}^{(i)}&=\theta_{n-1}^{(i)}+\gamma_{n} \sigma_{n}^{(i)} m_{n}^{(i)} \\
&=\theta_{n-1}^{(i)}+\gamma_{n}\bigl((H\theta_{n-1})^{(i)}+\underbrace{f^{(i)}(\theta_{n-1})-(H\theta_{n-1})^{(i)}}_{=:R_{n}^{(i)}}+\underbrace{g^{(i)}(\bX(n)) -f^{(i)}(\theta_{n-1})}_{=:\Upsilon^{(i)}_{n}}\\
&\qquad\qquad\qquad\qquad +\underbrace{ \sigma_{n}^{(i)} m_{n}^{(i)}-g^{(i)}(\bX(n))}_{=:\bar R_{n}^{(i)}}\bigr).
}
We denote for $m,n\in\N_{0}$ with $m\le n$,  $\Pi[m,n]=\prod_{k=m+1}^{n}(\1+\gamma_{k}H)$. Then one has on the event 
\bas{
\theta_{n}=\Pi[n_{0},n] \theta_{n_{0}} +\sum_{k=n_{0}+1}^{n} \gamma_{k}\,\Pi[k,n](\Upsilon_{k}+R_{k}+\bar R_{k}) 
}
and in terms of
\bas{
\bar\Pi[k,n]= \gamma_{k}\sum_{\ell=k}^{n}\Pi[k,\ell]\qquad (1\le k\le n)
}
that
\bas{\label{eq28347}
\bar\theta_{n}&=\frac 1{n-n_{0}}\sum_{\ell=n_{0}+1}^{n} \theta_{\ell}=\frac 1{n-n_{0}}\sum_{\ell=n_{0}+1}^{n}\Bigl( \Pi[n_{0},\ell] \,\theta_{n_{0}} +\sum_{k=n_{0}+1}^{\ell} \gamma_{k}\Pi[k,\ell](\Upsilon_{k}+R_{k}+\bar R_{k}) \Bigr)\\
&=\frac 1{n-n_{0}}\Bigl(\sum_{\ell=n_{0}+1}^{n} \Pi[n_{0},\ell] \,\theta_{n_{0}} +\sum_{k=n_{0}+1}^{n}\gamma_{k}\sum_{\ell=k}^{n}\Pi[k,\ell](\Upsilon_{k}+R_{k}+\bar R_{k}) \Bigr)\\
&=\frac 1{n-n_{0}}\Bigl(\bar\Pi[n_{0}+1,n] \,\underbrace{\gamma_{n_0+1}^{-1}\Pi[n_0,n_0+1]\theta_{n_{0}}}_{=:\tilde \theta_{n_0}} +\sum_{k=n_{0}+1}^{n}\bar\Pi[k,n](\Upsilon_{k}+R_{k}+\bar R_{k}) \Bigr).
}
We use the martingale correction to represent $\Upsilon_{k}$: Let for $k=n_0+1,\dots$
\bas{
\Delta M_k&=1_{\{\theta_{k-1}\in V\}}\bigl( \Upsilon_k+\Psi_{\theta_{k-1}}(\bX(k))-\Psi_{\theta_{k-1}}(\bX(k-1))\bigr)\\
&=1_{\{\theta_{k-1}\in V\}}\bigl(g(\bX(k))-f(\theta_{k-1}) +\Psi_{\theta_{k-1}}(\bX(k))-\Psi_{\theta_{k-1}}(\bX(k-1))\bigr)
}
and recall that by \cref{prop:3241} it is a sequence of  martingale differences.
We will compare it with the martingale differences \bas{\Delta \tilde M_k=g(\tilde \bX(k))+\Psi_{\theta^*}(\tilde \bX(k))-\Psi_{\theta^*}(\tilde \bX(k-1)),
}
where $\tilde \bX(k)=(\tilde X_{k+\ell})_{\ell\in-\N_0}=(X(U_{k+\ell},\theta^*))_{\ell\in-\N_0}$. 
Letting 
$\tilde R_{k}:=\Upsilon_k-\Delta M_k$ 
and $R_k':=\Delta M_k-\Delta\tilde  M_k$ we get with \cref{eq28347} that
\bas{\label{eq972356}
\bar\theta_{n}&=\frac 1{n-n_{0}}\Bigl(\bar\Pi[n_{0}+1,n] \,\tilde \theta_{n_{0}} +\sum_{k=n_{0}+1}^{n} \bar\Pi[k,n](\Delta\tilde M_{k}+R_{k}+\bar R_{k}+\tilde R_{k}+R_k') \Bigr).
}
%
%

We will show that the contribution of the quantities $R_{k}$, $\bar R_{k}$, $\tilde R_{k}$ and $R'_k$ is   asymptotically negligible (meaning that their overall contribution is of order  $o(1/\sqrt n)$, in probability, on $\Omega_0=\{\lim_{n\to\infty}\theta_n=\theta^*\}$).
Indeed, once we verified that 
\bas{\label{eq:923567}
\lim_{n\to\infty} \1_{\{\lim_{n\to\infty}\theta_n=\theta^*\}}\sqrt n  \frac 1{n-n_{0}}\sum_{k=n_{0}+1}^{n} \bar\Pi[k,n](R_{k}+\bar R_{k}+\tilde R_{k}+R_k') =0,\text{ in probability,}
}
we can combine this with \cref{prop:98734456} and a Lemma of Slutsky for stable convergence to get with \cref{eq972356} that
 \bas{
 \sqrt n \, \bar\theta_{n} \stackrel {\mathrm{stably}}\Longrightarrow \cN(0, H^{-1}\Gamma(H^{-1})^\dagger), \text{ \ \ on }\{\lim_{n\to\infty}\theta_n=\theta^*\},
 }
 where $H$ and $\Gamma$ are as in the theorem. This then finishes the proof.

 \emph{2. Localisation.}
We now fix $R,c_1\in(0,\infty)$ and a norm $\|\cdot\|$  on $\R^d$ induced by a scalar product $\llangle\cdot,\cdot\rrangle$ as in \cref{prop:23553} and \cref{prop:9873456}. For $n^*\in\N$ we consider the stopping time
\bas{
T_{n^*}:=\inf\bigl\{n\in\N\cap [n^*,\infty): \|\theta_n-\theta^*\|>R\bigr\}\wedge \inf \bigl\{n\in\N : | X(U_n,\theta^*)|\ge n^*(n^*+n)\bigr\}.
} 
and the events
\bas{\cE_{n^*}:=\{\theta_{n^*}\in \overline {\IB(\theta^*,R)} \}\cap \{\|\bX(n^{*})\|_{\ell_{\varrho}^{d}}\le (\gamma_{n^*})^{-1/2}\}\text{ \ and \ }
\Omega_{n^*}=\cE_{n^*}\cap \{ T_{n^*}=\infty\}.} By \cref{prop:9873456}, one has
\bas{\lim_{n^*\to\infty}\P(\{\lim_{n\to\infty}\theta_n=\theta^*\}\backslash \Omega_{n^*})=0.}
It thus suffices to prove \cref{eq:923567} for all sufficiently large $n^*\in\N$ with $\1_{\{\lim_{n\to\infty}\theta_n=\theta^*\}}$ replaced by $\1_{\Omega_{n^*}}$. In the following considerations $n^*\in\N$ is an arbitrarily fixed large index.

In order to show that the contributions of the terms $R_k$, $\bar R_k$, $\tilde R_k$, and $R'_k$ are asymptotically negligible, we perform a series of estimations. 
To simplify the presentation, we use the following convention: let $(a_n)$ and $(b_n)$ be real-valued expressions defined on an index set $I$. 
We write
\bas{
a_n \les b_n, \text{ \ for } n \in I
}
if and only if there exists a $\kappa \in (0,\infty)$ such that $a_n \le \kappa b_n$ for all $n \in I$. 
We use the same notation for families of random variables $(A_n)$ and $(B_n)$ and write
\bas{
A_n\prec B_n , \text{ \ for } n\in I
}
if and only if there exists a (deterministic) $\kappa\in(0,\infty)$ such that for all $n\in I$ one has $A_n\le \kappa B_n$.

\emph{3. Negligibility of $R_{k}$-term.}
By \cref{prop2352}, there exists $n_0^*\in\N$ and $\eta\in(0,\infty)$ such that for all $n^*\ge n_0^*$ one has
\bas{\label{eq:98734256}
\E[\1_{\cE_{n^*}}\1_{\{T_{n^*}\ge n\}} \|\theta_n-\theta^*\|^p]^{1/p} \le \eta \sqrt{\gamma_{n+1}}.
}
Fix $n^*\ge n_0^*$ arbitrarily.
As consequence of assumption~\cref{eq:237462} and the definition of $T_{n^*}$, one has
\bas{\1_{\{T_{n^*}\ge k\}} |R_k|=\1_{\{T_{n^*}\ge k\}} |f(\theta_{k-1})-H (\theta_{k-1}-\theta^*)|\les |\theta_{k-1}-\theta^*|^{1+\lambda}, \text{ \ for }k=n^*+1,\dots. 
}
 We combine this with \cref{eq:98734256} and assumption \cref{eq:237856-3} and conclude that
\bas{
\E\Bigl[\1_{\cE_{n^*}}\1_{\{T_{n^{*}}\ge n\}}\sum_{k=n^{*}+1}^{n}  |R_{k}|\Bigr]&\les \E\Bigl[\sum_{k=n^{*}+1}^{n} \1_{\cE_{n^*}}\1_{\{T_{n^{*}}\ge k\}} |\theta_{k-1}-\theta^*|^{1+\lambda}\Bigr]\\
&\les \sum_{k=n^{*}+1}^{n}  \sqrt{\gamma_{k}}^{1+\lambda} =o(\sqrt n), \text{ \ for }n\in\{n^*,n^*+1,\dots\}.
}
As consequence of the Markov inequality we have, on $\Omega_{n^{*}}=\cE_{n^*}\cap\{T_{n^{*}}=\infty\}$, that
\bas{
\lim_{n\to\infty}\sqrt n\,\frac 1{n-n_{0}} \sum_{k=n_{0}+1}^{n} |R_{k}| =0, \text{ in probability.}
}
By the uniform boundedness of the matrices $(\bar \Pi_{m,n})_{1\le m\le n}$ and the triangle inequality it follows that 
\bas{
\lim_{n\to\infty}\sqrt n\, \Bigl|\frac1{n-n_{0}} \sum_{k=n_{0}+1}^{n} \bar \Pi[k,n] R_{k}\Bigr| =0, \text{ in probability, on $\Omega_{n^*}$.}
}

\emph{4. Negligibility of $\bar R_{k}$-term.} Let $n\in\N$ with $n\ge n_0$ and $i\in\{1,\dots,d\}$.
Note that for $\tilde\sigma_n^{(i)}:=\frac 1{\sqrt {v_{n}^{(i)}}+\epsilon}$, one has
\bas{
\sqrt{1-\beta^{n}}\, \tilde\sigma_{n}^{(i)}\le \sigma_{n}^{(i)}=\frac 1{\sqrt{v_{n}^{(i)}/(1-\beta^{n})}+\epsilon}\le \tilde \sigma_{n}^{(i)}.
}
Consequently, using that $\tilde \sigma_{n}^{(i)}m_{n}^{(i)}=g^{(i)}(\bX(n))$, the uniform boundedness of $g$ and that for $y\in[0,1]$, $\sqrt y\ge y$, we get that
\bas{
|\bar R_n^{(i)}|&=|\sigma_n^{(i)} m_n^{(i)} -\tilde \sigma_{n}^{(i)}m_{n}^{(i)}|\le (1-\sqrt{1-\beta^{n}}) |\tilde \sigma_{n}^{(i)}m_{n}^{(i)}|\\
&=  (1-\sqrt{1-\beta^{n}})  \,|g(\bX^{(i)}(n))| \le \frac{1-\alpha}{\sqrt{1-\beta}\sqrt{1-\alpha^2/\beta}} \beta^n.
}
Together with the uniform boundedness of the matrices $(\bar\Pi[m,n])_{1\le m\le n}$ we conclude that
\bas{
\lim_{n\to\infty}\sqrt n\, \Bigl|\frac1{n-n_{0}} \sum_{k=n_{0}+1}^{n} \bar \Pi[k,n] \bar R_{k}\Bigr| =0.
}

\emph{5. Negligibility of $\tilde R_{k}$-term.} Again we fix $n^*\ge n_0$. Note that for all $k>n^*$ we have on $\{T_{n^*}\ge k\}$ that 
\bas{
\tilde R_k=\Upsilon_k-\Delta M_k=\Psi_{\theta_{k-1}}(\bX(k-1))- \Psi_{\theta_{k-1}}(\bX(k)).
}
Consequently, on $\Omega_{n^*}$, we have for all $n=n^*+1,\dots$ that
\bas{\label{eq87234}
\sum_{k=n^*+1}^n \bar\Pi[k,n]\, \tilde R_k&=\sum_{k=n^*}^{n-1} \bar\Pi[k+1,n]\, \Psi_{\theta_k}(\bX(k))-\sum_{k=n^*+1}^{n} \bar\Pi[k,n]\, \Psi_{\theta_{k-1}}(\bX(k))\\
&=\sum_{k=n^*+1}^{n-1} \bigl(\bar\Pi[k+1,n]\, \Psi_{\theta_k}(\bX(k))-\bar\Pi[k,n]\, \Psi_{\theta_{k-1}}(\bX(k))\bigr)\\
&\qquad +\bar\Pi[n^*+1,n]\, \Psi_{\theta_{n^*}}(\bX(n^*))-\bar\Pi[n,n]\, \Psi_{\theta_{n-1}}(\bX(n)).
}
By definition of $\cE_{n^*}$ and $T_{n^*}$ we get that 
\bas{\label{eq8532}
\E\bigl[\1_{\cE_{n^*}}&\1_{\{T_{n^*}\ge k\}} \|\bX(k)\|^2_{\ell_\varrho^d}\bigr]^{1/2}\le  (\sqrt{\beta})^{k-n^*} \E[\1_{\cE_{n^*}} \|\bX(n^*)\|_{\ell_\varrho^d}^2]^{1/2}+ \sum_{r=n^*+1}^k \varrho_{r-k} \,\E[\1_{\{T_{n^*}\ge r\}} |X(\theta_{r-1},U_r)|^2]^{1/2}\\
&\le \gamma_{n^*}^{-1/2} (\sqrt\beta)^{k-n^*}+ \|\varrho\|_{\ell_1}\,\cK \les 1, \text{ \ for }k\in\{n^*,n^*+1,\dots\}. 
}
In combination with the uniform boundedness of the $\bar\Pi$-matrices and estimate \cref{eq:8761245} we get that
\bas{
\E[\1_{\cE_{n^*}}\1_{\{T_{n^*}\ge n\}}  |\bar\Pi[n,n] \Psi_{\theta_{n-1}}(\bX(n))|]\les \E\bigl[
\1_{\cE_{n^*}}\1_{\{T_{n^*}\ge n\}} \bigl(\|\bX(n)\|_{\ell_{\varrho}^{d}}+ \E[|X(\theta,U)|]\big|_{\theta=\theta_{n-1}}\bigr)\bigr]\les 1,
}
for $n=n^*+1,\dots$. 
Using the  boundedness of the $\bar\Pi$-operators we get that altogether, on $\Omega_{n^*}$, in probability
\bas{\label{eq8735632}
\lim_{n\to\infty}\sqrt n \,\Bigl| \frac 1{n-n_0} \Bigl(&\sum_{k=n_0+1}^{n^*} \bar\Pi[k,n]\, \tilde R_k + \bar\Pi[n^*+1,n]\, \Psi_{\theta_{n^*}}(\bX(n^*)-\bar \Pi[n,n]\,\Psi_{\theta_{n-1}}(\bX(n))\Bigr)\Bigr|=0.
}
It remains to analyse the contribution of the summands with index $k\in\{n^*+1,\dots,n\}$ on the right-hand side of \cref{eq87234}.
Note that on $\Omega_{n^*}$ one has for  all $k=n^*+1,\dots,n-1$ that
\bas{\label{eq765214}
\bigl |\bar\Pi[k+1,n]\, \Psi_{\theta_k}(\bX(k))-\bar\Pi[k,n]\, \Psi_{\theta_{k-1}}(\bX(k))\bigr | &\le \bigl |\bar\Pi[k+1,n]-\bar\Pi[k,n] \bigr| \,|\Psi_{\theta_k}(\bX(k))|\\
&\qquad + \bigl|\bar\Pi[k,n] \bigr| \,|\Psi_{\theta_k}(\bX(k))-\Psi_{\theta_{k-1}}(\bX(k))|
}
and that
\bas{\label{eq874356}
\bar\Pi[k+1,n]-\bar\Pi[k,n]&=\gamma_{k+1} \sum_{\ell=k+1}^n\Pi[k+1,\ell]-\gamma_k  \sum_{\ell=k}^n\Pi[k,\ell]\\
&=-\gamma_k \II+\sum_{\ell=k+1}^n (\gamma_{k+1}\II-\gamma_k \II-\gamma_k\gamma_{k+1} H) \,\Pi[k+1,\ell]\\
&=-\gamma_k \II+ \Bigl(\Bigl(1-\frac{\gamma_k}{\gamma_{k+1}}\Bigr) \II - \gamma_k H\Bigr) \,\bar \Pi[k+1,n].
}
By assumption~\cref{eq:237856-1}, we have that  $\lim_{k\to\infty} (1-\frac{\gamma_{k+1}}{\gamma_k})/\gamma_k=0$ so that, in particular, $\lim_{k\to\infty}\gamma_{k+1}/\gamma_k=1$. This entails that the expression
\bas{
\gamma_k^{-1} \Bigl(1-\frac{\gamma_k}{\gamma_{k+1}}\Bigr)=\frac {\gamma_{k+1}-\gamma_k}{\gamma_k^2}\frac{\gamma_k}{\gamma_{k+1}}
}
tends to zero and is, in particular,  uniformly bounded over all $k\in\N$. Together with~\cref{eq874356} and the uniform boundedness of the matrices $(\bar\Pi[m,n])_{1\le m\le n}$ we conclude that
\bas{\label{eq765214-2}
\bigl |\bar\Pi[k+1,n]-\bar\Pi[k,n] \bigr|
\les \gamma_k\text{, \ for  }(k,n)\in\{(k',n') \in\N^2: k'\le n'\}.}
Combining this estimate with the uniform boundedness of $(\bar\Pi[m,n])_{1\le m\le n}$ and~\cref{eq765214} we get that, on $\Omega_{n^*}$,
\bas{\label{eq:862461}
\bigl |\bar\Pi[k+1,n]\, \Psi_{\theta_k}(\bX(k))-\bar\Pi[k,n]\, \Psi_{\theta_{k-1}}(\bX(k))\bigr | \les  \gamma_k \,|\Psi_{\theta_k}(\bX(k))|+ \,|\Psi_{\theta_k}(\bX(k))-\Psi_{\theta_{k-1}}(\bX(k))|,
}
for $(k,n)\in\{(k',n') \in\N^2: k'\le n'\}$. On $\{T_{n^*}> k\}$, the state $\theta_k$ is in $\overline{\IB(\theta^*,R)}\subset V$ and we conclude with \cref{le:89734556} and the uniform boundedness of the moments of the innovation on $\overline{\IB(\theta^*,R)}$ that, on $\Omega_{n^*}$,
\bas{\label{eq:783456}
|\Psi_{\theta_k}(\bX(k))|\les  \|\bX(k)\|_{\ell_{\varrho}^d}+1, \text{ \ for }k\in\{n^*,\dots\}.
}
Moreover, by \cref{le:3451}, one has that, on $\Omega_{n^*}$,
\bas{
|&\Psi_{\theta_k}(\bX(k))-\Psi_{\theta_{k-1}}(\bX(k))|\\
&\les \sum_{i=1}^d \bigl(1+(\delta^{-1/2}+v^*(\bX^{(i)}(k))^{-1/2})(\|\bX^{(i)}(k)\|_{\ell_{\varrho}}
+ \E[|X^{(i)}(\theta,U)|^{2}]\big|_{\theta=\theta_{k-1}}^{1/2}) \bigr) \\
& \qquad \qquad \bigl(\E[|X(\theta,U)-X(\theta',U)|^{2}]\big|_{(\theta,\theta')=(\theta_k,\theta_{k-1})}\bigr)^{1/2}\\
&\les \sum_{i=1}^d  (v^*(\bX^{(i)}(k))^{-1/2}+1) (\|\bX(k)\|_{\ell_{\varrho}^d} +1)|\theta_k-\theta_{k-1}|,\text{ \ for }k\in\{n^*+1,\dots\}.
}
Using that $\theta_n^{(i)}-\theta^{(i)}_{n-1}=\gamma_n \sigma_n^{(i)}m_n^{(i)}$ and $|\sigma_n^{(i)}m_n^{(i)}|\le |g^{(i)}(\bX(n))|\le \frac{1-\alpha}{\sqrt{1-\beta}\sqrt{1-\alpha^2/\beta}} $ we conclude that, on $\Omega_{n^*}$,
\bas{\label{eq873563}
|\Psi_{\theta_k}(\bX(k))-\Psi_{\theta_{k-1}}(\bX(k))|\les \sum_{i=1}^d  (v^*(\bX^{(i)}(k))^{-1/2}+1) (\|\bX(k)\|_{\ell_{\varrho}^d} +1) \,  \gamma_k,
}
for $k\in\{n^*+1,\dots\}$.
 Together with \cref{eq:862461} and \cref{eq:783456} we get that, on $\Omega_{n^*}$,
\bas{\label{eq:82344566}
\bigl |\bar\Pi[k+1,n]\, \Psi_{\theta_k}(\bX(k))-\bar\Pi[k,n]\, \Psi_{\theta_{k-1}}(\bX(k))\bigr | \les  \sum_{i=1}^d  (v^*(\bX^{(i)}(k))^{-1/2}+1) (\|\bX(k)\|_{\ell_{\varrho}^d} +1) \,  \gamma_k,
}
for $(k,n)\in\{(k',n') \in\N^2: n^*<k'< n'\}$. We are now in the position to apply \cref{prop:7356}. We consider the events $(\cE'_k)_{k\in\N}$ as defined in \cref{eq:367122346-2} and note that all but finitely many of the events occur on $\Omega_{n^*}$.
Now using \cref{eq:82344566} and the Cauchy-Schwarz inequality in the first step and \cref{eq:367122346-2} and \cref{eq8532} in the second step we get that for $n\in\{n^*+1,\dots\}$
\bas{
\E\Bigl[&\sum_{k=n^*+1}^{n-1} \1_{\Omega_{n^*}}\1_{\cE'_k}  \bigl |\bar\Pi[k+1,n]\, \Psi_{\theta_k}(\bX(k))-\bar\Pi[k,n]\, \Psi_{\theta_{k-1}}(\bX(k))\bigr |\Bigr]\\
&\les \sum_{k=n^*+1}^{n-1}  \gamma_k\, \E\Bigl[\1_{\Omega_{n^*}}\1_{\cE'_k} \sum_{i=1}^d  v^*(\bX^{(i)}(k))^{-1}+1\Bigr]^{1/2}  \E\Bigl [ \1_{\Omega_{n^*}}\bigl(\|\bX(k)\|_{\ell_{\varrho}^d} +1\bigr)^2\Bigr]^{1/2} \\&\les \sum_{k=n^*+1}^{n-1}  \gamma_k.
}
As consequence of \cref{eq:237856-3} we have that $\lim_{n\to\infty}\frac1{\sqrt n}\sum_{k=n^*+1}^{n-1}\gamma_k=0$ so that  by the Markov inequality, in probability, on $\Omega_{n^*}$,
\bas{
 \lim_{n\to\infty} \sqrt n  \frac 1{n-n_0}\sum_{k=n^*+1}^{n-1}\bigl |\bar\Pi[k+1,n]\, \Psi_{\theta_k}(\bX(k))-\bar\Pi[k,n]\, \Psi_{\theta_{k-1}}(\bX(k))\bigr|=0.
}
Together with \cref{eq8735632} we thus showed negligibility of the $\tilde R_k$-term.

\emph{6. Neglibility of $R_k'$-term.}
It remains to show that, in probability, on $\Omega_{n^*}$,
\bas{
\lim_{n\to\infty}\sqrt n\,\frac1{n-n_0}\sum_{k=n_0+1}^n \bar\Pi[k,n]\,R_k'=0,
}
where $R_k'=\Delta M_k-\Delta\tilde M_k$,
\bas{
\Delta M_k=\1_{\{\theta_{k-1}\in V\}}\bigl(\Psi_{\theta_{k-1}}(\bX(k))-\Psi_{\theta_{k-1}}(\bX(k-1))+g(\bX(k))-f(\theta_{k-1})\bigr)
}
and
\bas{
\Delta \tilde M_k=\Psi_{\theta^*}(\tilde \bX(k))-\Psi_{\theta^*}(\tilde \bX(k-1))+g(\tilde \bX(k)).
}
Recall that by \cref{eq8532}, $\E\bigl[\1_{\cE_{n^*}}\1_{\{T_{n^*}\ge k\}} \|\bX(k)\|_{\ell_\varrho^d}\bigr]$ is finite and that $f(\theta)$ and $\E[|X(U,\theta)|]$ are uniformly bounded over all $\theta\in\overline{\IB(\theta^*,R)}$.
Hence, by \cref{le:89734556}, $\1_{\cE_{n^*}}\1_{\{T_{n^*}\ge k\}}\Delta M_k$ is integrable and \cref{prop:3241} implies that,  for all $k=n^*+1,\dots$,
\bas{
\E[\1_{\cE_{n^*}}\1_{\{T_{n^*}\ge k\}}\Delta M_k|\cF_k]=0.
}
By the same proposition the same is true when replacing $\Delta M_k$ by $\Delta \tilde M_k$ so that $\E[\1_{\cE_{n^*}}\1_{\{T_{n^*}\ge k\}} R_k'|\cF_{k-1}]=0$.

As in \cref{prop:7356} we
let for $k\in\N$, $r(k)=\lfloor k-(\log k)^2\rfloor$,
\bas{\label{eq783549061-2}
\tilde \cE'_k=\{\forall i\in\{1,\dots,d\}\exists \text{distinct }&\ell_1,\ell_2\in\{r(k),\dots, k-1\}\text{ with }\\
&(X^{(i)}(\theta^*,U_{\ell_1}))^2\wedge (X^{(i)}(\theta^*,U_{\ell_2}))^2\ge \delta\}
}
and
\bas{\label{eq783549061-4}
\tilde \cE''_k=\bigl\{\bigl\|(\1_{\{k+\ell<r(k)\}}\,X(\theta^*,U_{k+\ell}))_{\ell\in-\N_0}\bigr\|_{\ell_\varrho^d}\le 1\bigr\}.
}
We choose $n^{**}\in\N$ such that for all $k\ge n^{**}+1$ one has $r(k)>n^*$. By \cref{prop:7356}, almost surely, all but finitely many of the events $(\tilde \cE'_k)_{k\in\N}$ and $(\tilde \cE''_k)_{k\in\N}$  occur and we conclude that by the uniform boundedness of the $\bar\Pi$-matrices one has, in probability, on $\Omega_{n^*}$,
\bas{\label{eq124773}
\lim_{n\to\infty}\sqrt n\,\Bigl|\frac1{ n-n_0} \sum_{k=n_0+1}^{n} \1_{\{k\le n^{**}\}\cup  (\tilde \cE_k')^c\cup  (\tilde \cE_k'')^c}\, \bar\Pi[k,n] R_k' \Bigr|=0.
}
It remains to show that the contribution of all summands  $k>n^{**}$ for which the event $\tilde \cE'_k$ and $\tilde\cE_k''$ enter is also negligible.  By definition the events $(\tilde \cE'_k)_{k>n^{**}}$ and $(\tilde \cE''_k)_{k>n^{**}}$ are predictable so that for every fixed $n>n^{**}$ also $(\1_{\cE_{n^*}}\1_{\tilde \cE'_k}\1_{\tilde \cE''_k}\1_{\{T_{n^*}\ge k\}}\bar \Pi[k,n] R_k')_{k=n^{**}+1,\dots n}$ forms a sequence of  martingale differences. Together with the  
 uniform boundedness of $(\bar\Pi[k,n])$ we get that
\bas{\label{eq893547}
\E\Bigl[\Bigl| \sum_{k=n^{**}+1}^n &\1_{\cE_{n^*}}\1_{\tilde \cE'_{k}}\1_{\tilde \cE''_k}\1_{\{T_{n^*}\ge k\}}\bar\Pi[k,n] R_k'\Bigr|^2 \Big]\\
&=\sum_{k=n^{**}+1}^n \E\bigl[\bigl|\1_{\cE_{n^*}}\1_{\tilde \cE'_{k}}\1_{\tilde \cE''_k}\1_{\{T_{n^*}\ge k\}}\bar\Pi[k,n]R_k'\bigr|^2 \big]\\
&\les \,\sum_{k=n^{**}+1}^n \E\bigl[\1_{\cE_{n^*}}\1_{\tilde \cE'_{k}}\1_{\tilde \cE''_k}\1_{\{T_{n^*}\ge k\}}\bigl|R_k'\bigr|^2 \big]\text{ \ for }n\in\{n^{**}+1,\dots\}.
}
We will provide an estimate for the latter expectations.
Using the triangle inequality  we get that on $\{T_{n^*}\ge k\}$
\bas{\label{eq87241}
|R_k'|&\le |\Psi_{\theta_{k-1}}(\bX(k))- \Psi_{\theta_{k-1}}(\tilde \bX(k))| +| \Psi_{\theta_{k-1}}(\tilde \bX(k))-  \Psi_{\theta^*}(\tilde \bX(k))|\\
&\quad+  |\Psi_{\theta_{k-1}}(\bX(k-1))- \Psi_{\theta_{k-1}}(\tilde \bX(k-1))| +| \Psi_{\theta_{k-1}}(\tilde \bX(k-1))-  \Psi_{\theta^*}(\tilde \bX(k-1))|\\
&\quad + |g(\bX(k))-g(\tilde \bX(k))|+ |f(\theta_{k-1})|.
}
First note that the first, third, fifth and sixth term can be bounded by using \cref{le:523}, the Lipschitz continuity of $f$ on $B(\theta^*,R)$ and that $\|\bX(k-1)-\tilde \bX(k-1)\|_{\ell_\varrho^d}\le \mathrm{const}\, \|\bX(k)-\tilde \bX(k)\|_{\ell_\varrho^d}$ for a constant $\mathrm{const}$ only depending on $\alpha$ and $\beta$:
\bas{\label{eq87241-2}
\E\Bigl[\1_{\cE_{n^*}}&\1_{\tilde \cE'_{k}}\1_{\tilde \cE''_k}\1_{\{T_{n^*}\ge k\}}
\Bigl(|\Psi_{\theta_{k-1}}(\bX(k))- \Psi_{\theta_{k-1}}(\tilde \bX(k))| + |\Psi_{\theta_{k-1}}(\bX(k-1))- \Psi_{\theta_{k-1}}(\tilde \bX(k-1))| \\
&\qquad\qquad\qquad\qquad + |g(\bX(k))-g(\tilde \bX(k))|+ |f(\theta_{k-1})|\Bigr)^2
\Bigr]\\
&\les \E\bigl[\1_{\cE_{n^*}}\1_{\{T_{n^*}\ge k\}}  \bigl(\|\bX(k)-\tilde \bX(k)\|_{\ell_\varrho^d}^2+  |\theta_{k-1}-\theta^*|^2\bigr)\bigr] .
}
One has for every $k=n^*+1,\dots$ that
\bas{
\|\bX(k)-\tilde \bX(k)\|^2_{\ell_\varrho^d}&=\Bigl(\sum_{\ell\in\Z\cap (-\infty,k]} \varrho_{\ell-k}|X_\ell-\tilde X_\ell|\Bigr)^2\\
&\le \Bigl(\sum_{\ell\in\Z\cap (-\infty,n^*]} \varrho_{\ell-k}|X_\ell|+ \sum_{\ell\in\Z\cap (-\infty,n^*]} \varrho_{\ell-k}|\tilde X_\ell| + \sum_{\ell=n^*+1}^k \varrho_{\ell-k}|X_\ell-\tilde X_\ell|\Bigr)^2\\
&\le 2 \Bigl(\sum_{\ell\in\Z\cap (-\infty,n^*]} \varrho_{\ell-k}|X_\ell|\Bigr)^2+ 2 \|\varrho\|_{\ell_1}  \Bigl( \sum_{\ell\in\Z\cap (-\infty,n^*]} \varrho_{\ell-k}|\tilde X_\ell|^2 + \sum_{\ell=n^*+1}^k \varrho_{\ell-k}|X_\ell-\tilde X_\ell|^2\Bigr) \\
&\le 2 \beta^{k-n^*} \|\bX(n^*)\|_{\ell_\varrho^d}^2+ 2\|\varrho\|_{\ell_1} \Bigl( \sqrt\beta^{k-n^*}\sum_{\ell\in\Z\cap (-\infty,n^*]} \varrho_{\ell-n^*}|\tilde X_\ell|^2 + \sum_{\ell=n^*+1}^k \varrho_{\ell-k}|X_\ell-\tilde X_\ell|^2\Bigr) .
}
Using that $\1_{\cE_{n^*}} \|\bX(n^*)\|_{\ell_\varrho^d}\le \gamma_{n^*}^{-1/2}$, $\E[|\tilde X_\ell|^2]\le \cK^2$ for all $\ell\in\Z$ we conclude that
\bas{
\E\bigl[ \1_{\cE_{n^*}}\1_{\{T_{n^*}\ge k\}}\|\bX(k)-\tilde \bX(k)\|_{\ell_\varrho^d}^2\bigr]&\les \sqrt\beta^{k-n^*}+ \sum_{\ell=n^*+1}^k \varrho_{\ell-k} \,\E[ \1_{\cE_{n^*}}\1_{\{T_{n^*}\ge \ell\}} |X_\ell-\tilde X_\ell|^2] 
}
for $k\in\{n^*+1,\dots\}$. Next, we combine this estimate with 
 $\E[\1_{\cE_{n^*}}\1_{\{T_{n^*}\ge \ell\}} |X_\ell-\tilde X_\ell|^2|\cF_{\ell-1}]\le \cK^2 |\theta_{\ell-1}-\theta^*|^2 $ and 
$\E[\1_{\cE_{n^*}}\1_{\{T_{n^*}\ge \ell\}}|\theta_{\ell-1}-\theta^*|^2]\les \gamma_{\ell-1}$ both for $\ell\in\{n^*+1,\dots\}$
and get that
\bas{
\sum_{k=n^*+1}^n \E\bigl[ \1_{\cE_{n^*}}\1_{\{T_{n^*}\ge k\}}\|\bX(k)-\tilde \bX(k)\|_{\ell_\varrho^d}^2\bigr]&\les \sum_{k=n^*+1}^n\Bigl( \sqrt\beta^{k-n^*}+\sum_{\ell=n^*+1}^k \varrho_{\ell-k} \,\E[ \1_{\cE_{n^*}}\1_{\{T_{n^*}\ge \ell\}} |\theta_{\ell-1}-\theta^*|^2]\Bigr)\\
&\le (1-\sqrt\beta)^{-1}+ \sum_{\ell=n^*+1}^n  \sum_{k=\ell}^n\varrho_{\ell-k} \,\E[ \1_{\cE_{n^*}}\1_{\{T_{n^*}\ge \ell\}} |\theta_{\ell-1}-\theta^*|^2]\\
&\les 1+ \sum_{\ell=n^*+1}^n \gamma_{\ell-1} \text{ \ for }n\in\{n^*,\dots\}.
}
Assumption~\cref{eq:237856-3} implies that $\sum_{k=1}^n \gamma_k=o(\sqrt n)$ so that
\bas{\label{eq735433}
\lim_{n\to\infty} n^{-1/2} \sum_{k=n^{**}+1}^n \E\bigl[ \1_{\cE_{n^*}}\1_{\{T_{n^*}\ge k\}}(\|\bX(k)-\tilde \bX(k)\|_{\ell_\varrho^d}^2+|\theta_{k-1}-\theta^*|^2\bigr)\bigr]=0.
}
Together with \cref{eq87241-2} this shows that  the first, third, fifth and sixth term in~\cref{eq87241} are asymptotically negligible. It remains to analyse the contribution of the second and fourth term.
For this either  let
\bas{
\bY(k)=\tilde \bX(k)\text{ \ \ or \ \ } \bY(k)=\tilde \bX(k-1)
} 
for all $k\in\N$. Using \cref{le:3451} we conclude that
that for $k\in\{n^*+1,\dots\}$
\bas{\label{eq87241-3}
\E\bigl[&\1_{\cE_{n^*}}\1_{\tilde \cE'_{k}}\1_{\tilde \cE''_k}\1_{\{T_{n^*}\ge k\}}
| \Psi_{\theta_{k-1}}(\bY(k))-  \Psi_{\theta^*}(\bY(k))|^2
\bigr]\\
&\les \E\Bigl[ \1_{\cE_{n^*}}\1_{\tilde \cE'_{k}}\1_{\tilde \cE''_k}\1_{\{T_{n^*}\ge k\}} \sum_{i=1}^d \Bigl(1+ \bigl(\delta^{-1/2}+v^{*}(\bY^{(i)}(k))^{-1/2}\bigr)\bigl(\|\bY^{(i)}(k)\|_{\ell_{\varrho}}+\E[X^{(i)}(\theta^*,U)^2]^{1/2}\bigr)\Bigr)^2 \\
&\qquad\qquad\qquad\qquad \qquad \cdot \E[|X^{(i)}(\theta^*,U)-X^{(i)}(\theta',U)|^{2}]\big|_{\theta'=\theta_{k-1}} \Bigr]\\
&\les  \E\Bigl[ \1_{\cE_{n^*}}\1_{\tilde \cE'_{k}}\1_{\tilde \cE''_k}\1_{\{T_{n^*}\ge k\}} \sum_{i=1}^d \bigl(1+v^{*}(\bY^{(i)}(k))^{-1}\bigr)\bigl(\| \bY^{(i)}(k)\|^2_{\ell_{\varrho}^d}+1\bigr) \, |\theta_{k-1}-\theta^*|^2  \Bigr]\\
&\les  \E\Bigl[ \sum_{i=1}^d \bigl(1+v^{*}(\bY^{(i)}(k))^{-1}\bigr)\bigl(\| \bY^{(i)}(k)\|^2_{\ell_{\varrho}^d}+1\bigr) \, |\theta_{k-1}-\theta_{r(k)-1}|^2  \Bigr]\\
&\qquad + \E\Bigl[ \1_{\cE_{n^*}}\1_{\tilde \cE'_{k}}\1_{\tilde \cE''_k}\1_{\{T_{n^*}\ge k\}} \sum_{i=1}^d \bigl(1+v^{*}(\bY^{(i)}(k))^{-1}\bigr)\bigl(\| \bY^{(i)}(k)\|^2_{\ell_{\varrho}^d}+1\bigr) \, |\theta_{r(k)-1}-\theta^*|^2  \Bigr].
}
We note that for a constant $c$  only depending on $\alpha,\beta, \epsilon$ one has that, for all $k>n^{**}$, 
\bas{|\theta_{r(k)-1}-\theta_{k-1}|\le c (k-r(k))\gamma_{r(k)}\le c(1+(\log k)^2)\gamma_{r(k)}=o(\sqrt {\gamma_{k-1}})} as $k\to\infty$, where we used \cref{eq:683461} of \cref{prop:7356} in the latter step.
Moreover, by the second statement of \cref{le:345143-2} one has that for $k=n^{**}+1,\dots$
\bas{
\E\bigl[ & \bigl(1+v^{*}(\bY^{(i)}(k))^{-1}\bigr)\bigl(\| \bY^{(i)}(k)\|^2_{\ell_{\varrho}^d}+1\bigr)\bigr]  \les 1.}
 This proves that the first term on the right-hand side of \cref{eq87241-3} is of order $o(\gamma_{k-1})$.
 
Next, we analyse the second term. One has 
\bas{\label{eq:2354674}
\E\Bigl[ &\1_{\cE_{n^*}}\1_{\tilde \cE'_{k}}\1_{\tilde \cE''_k}\1_{\{T_{n^*}\ge k\}} \sum_{i=1}^d \bigl(1+v^{*}(\bY^{(i)}(k))^{-1}\bigr)\bigl(\| \bY^{(i)}(k)\|^2_{\ell_{\varrho}}+1\bigr) \, |\theta_{r(k)-1}-\theta^*|^2  \Bigr]\\
&\le \E\Bigl[ \1_{\cE_{n^*}}\1_{\{T_{n^*}\ge r(k)-1\}} \1_{\tilde \cE''_k}\E\Bigl[ \1_{\tilde \cE'_{k}}\sum_{i=1}^d \bigl(1+v^{*}(\bY^{(i)}(k))^{-1}\bigr)\bigl(\| \bY^{(i)}(k)\|^2_{\ell_{\varrho}}+1\bigr) \Big|\cF_{r(k)-1}\Bigr]\, |\theta_{r(k)-1}-\theta^*|^2  \Bigr].
}
Now we write ${\tilde \bY}^{(i)}(k)$ for the  vector that is constituted by all entries of $\bY^{(i)}(k)$ which belong to $X$-values with $U$-index in $-\N_0\cup\{r(k),\dots,k\}$ with the remaining components being set to zero. Moreover, we write  $\tilde {\tilde \bY}^{(i)}(k)$ for the remaining entries so that we have 
\bas{
\bY^{(i)}(k)={\tilde \bY}^{(i)}(k)+\tilde {\tilde \bY}^{(i)}(k).
}
By definition, ${\tilde \bY}^{(i)}(k)$ is independent of $\cF_{r(k)-1}=\sigma(U_1,\dots,U_{r(k)-1})$ and $\tilde {\tilde \bY}^{(i)}(k)$ is $\cF_{r(k)-1}$-measurable. We thus can estimate the inner conditional expectation by
\bas{
\E\Bigl[ \1_{\tilde \cE'_{k}}\sum_{i=1}^d \bigl(1+v^{*}(\bY^{(i)}(k))^{-1}\bigr)\bigl(\| \bY^{(i)}(k)\|^2_{\ell_{\varrho}}+1\bigr) \Big|\cF_{r(k)-1}\Bigr]
\\
\le  2 \E\Bigl[ \1_{\tilde \cE'_{k}}\sum_{i=1}^d \bigl(1+v^{*}(\tilde \bY^{(i)}(k))^{-1}\bigr)\bigl(\| \tilde \bY^{(i)}(k)\|^2_{\ell_{\varrho}}+1\bigr) \Bigr]\\
+2 \E\Bigl[ \1_{\tilde \cE'_{k}}\sum_{i=1}^d \bigl(1+v^{*}(\tilde \bY^{(i)}(k))^{-1}\bigr)\Bigr]\, \| \tilde{ \tilde \bY}^{(i)}(k)\|^2_{\ell_{\varrho}}
}
By \cref{le:345143-2}, both expectations are uniformly bounded over all $k\in\{n^{**}+1,\dots\}$ so that we get with \cref{eq:2354674} that
\bas{
\E\Bigl[ &\1_{\cE_{n^*}}\1_{\tilde \cE'_{k}}\1_{\tilde \cE''_k}\1_{\{T_{n^*}\ge k\}} \sum_{i=1}^d \bigl(1+v^{*}(\bY^{(i)}(k))^{-1}\bigr)\bigl(\| \bY^{(i)}(k)\|^2_{\ell_{\varrho}}+1\bigr) \, |\theta_{r(k)-1}-\theta^*|^2  \Bigr]\\
&\les \E\Bigl[ \1_{\cE_{n^*}}\1_{\tilde \cE''_k}\1_{\{T_{n^*}\ge r(k)-1\}}(1+  \| \tilde{ \tilde \bY}(k)\|^2_{\ell^d_{\varrho}}) \, |\theta_{r(k)-1}-\theta^*|^2  \Bigr]  \\
&\les \E\Bigl[ \1_{\cE_{n^*}}\1_{\{T_{n^*}\ge r(k)-1\}} \, |\theta_{r(k)-1}-\theta^*|^2  \Bigr] \les \gamma_{r(k)-1}\les \gamma_{k-1} \text{ \ for }k\in\{n^{**}+1,\dots\} .
}
Thus we get with \cref{eq87241-3} that the second and fourth term on the right-hand side of \cref{eq87241} are also negligible.
This finishes the proof of the negligibility of the $R_k'$-term. Hence, the proof is complete.

\subsubsection*{Acknowledgements}
This work has been partially funded by the National Science Foundation of China (NSFC) under grant number W2531010.
Moreover, this work has been partially funded by the Deutsche Forschungsgemeinschaft (DFG, German Research Foundation)
under Germany's Excellence Strategy EXC 2044--390685587, Mathematics Münster: Dynamics--Geometry--Structure.
Furthermore, this work has been partially funded by the European Union (ERC, MONTECARLO, 101045811).
The views and the opinions expressed in this work are however those of the authors only and do
not necessarily reflect those of the European Union or the European Research Council (ERC).
Neither the European Union nor the granting authority can be held responsible for them.

\bibliographystyle{plain}
\bibliography{bibfile}

\end{document}